\documentclass[a4paper,10pt]{amsart}

\usepackage{extensions}
\begin{document}
\title[]{Law of large numbers for stochastic multiscale spatial gene networks}

\author[A. Debussche, B. Huguet]{Arnaud Debussche, Baptiste Huguet} 
\address{Univ Rennes, CNRS, IRMAR - UMR 6625, F-35000 Rennes, France}
\email{baptiste.huguet@math.cnrs.fr}
\email{arnaud.debussche@ens-rennes.fr}
\urladdr{https://www.math.cnrs.fr/~bhuguet/}
\begin{abstract}
We study a stochastic multiscale spatial gene network. These naturally arise in molecular biology.  In our model, the reactants are subject to on-site reactions on both scales and diffusion on the continuous scale only, although diffusion on both scales could easily be handled. We obtain, under a light condition on the scales between the total population size and the mesh discretisation, the convergence of the stochastic system to a deterministic system consisting of a PDE coupled to a ODE. This is in contrast with the well-stirred case where jumps remain at the limit. In order to prove this convergence result, we develop some moments control for martingales in discrete Sobolev topologies and use products rule in discrete Sobolev spaces.
\end{abstract}
\maketitle
\setcounter{tocdepth}{1}
\tableofcontents

\section{Introduction}

In molecular biology, a gene regulatory network is the system of reactions which regulates the expression of genes, the mRNA transcription and the proteins' translation. These networks are modelled using Markov jump processes, building on Delbr\"uck's seminal work \cite{Delbr}. Since then, the stochastic nature of gene networks has been emphasized by experimental biology \cite{Kepler}. Specifically, stochastic models replicate experimentally observed behaviours, including the burst production of proteins \cite{Cai}, the phenotype variability in isogenic population \cite{Kaern}, the noise propagation \cite{Paul}, or the circadian clocks stability \cite{Bark}. However, Markov models implementation is computationally  challenging. Gillespie algorithm, also referred to as Stochastic Simulation Algorithm (SSA), is not effective for real systems with numerous reagents or for systems that are highly interconnected. On the other hand, the exact number of molecules of each reagent is irrelevant when dealing with large populations. We are more interested in concentrations, which can be thought of as the ratio of the population of a species divided by the total population. The convergence of Markov models, in a large population limit, has been first studied for a long time, especially by Kurtz, for homogeneous (or well-stirred) case. Kurtz proves a law of large numbers, a central limit theorem, and a diffusion approximation, in probability, for the uniform topology (\cite{Kurtz70}, \cite{Kurtz71} and \cite{Kurtz78}, see also \cite[Chapter 11]{EthKur} for simple proofs). This homogeneous mono-scale case is still of mathematical interest nowadays, for instance for the study of large deviation principles \cite{ParSK}. The convergence of non-homogeneous models, with spatial dependence, has been studied by Arnold and Thedosopulu \cite{Arnold80dlo}. The associated stochastic models have two kinds of jumps : reaction jumps (as in the homogeneous case), and diffusion jumps of a molecule to an adjacent site, with local rate.  They proved convergence under a high density assumption : $l/N^2\to\infty$, where $l$ denotes the total population scale per site and $1/N$ the size of the mesh. These results have been developed by Kotelenez (\cite{Kotelenez86lol} and \cite{Kotelenez87fnh}) and further improved by Blount (\cite{Blount91} and \cite{Blount92}), who used the regularisation properties of the (discrete) heat semigroup to relax the high density assumption. In particular, he proved the convergence in $L^\infty$ under the density assumption $l/\log(N)\to\infty$.

A particularity of gene networks, amongst other biological systems, is the coexistence of different orders of magnitude for the number of reactants. For some molecules in abundance, the concentration is the relevant information, but for others, in scarce quantity, the concentration is null. That means that we must renormalise the system with different scales, depending on the type of reactant. In the homogeneous - or well-stirred - setting, Crudu, the first author, and Radulescu introduce a hybrid piecewise deterministic Markov process (PDMP), coupling ODE and jump process, as a model for gene network in \cite{CDR}. In a second article \cite{CDMR} with Muller, they prove that the multi-scale renormalised Markov model converges in distribution to a PDMP in various realistic biologic situations. This result has been improved by the second author in \cite{Huguet25sas}, with a strong convergence and a central limit theorem in distribution. The case with spatial dependance has been investigated by the first author and Nguepedja Nankep. The stochastic multi-scale spatial models considered have two scales. The first one, the "continuous scale", has an abundant population (of order $l$), fast reactions (rates of order $l$), and diffusion (rates of order $lN^2$). The second one, the "discrete scale", has a rare population (of order $1$), with slow reaction (rates of order $1$) and no diffusion.  The biological assumption, also used for mono-scale models, is the "one-site reaction" : a jump occurs in a mesh with a rate depending on the population of this mesh only. Under this assumption, the discrete scale presents a real obstruction to convergence. Indeed, the discrete part of the limit process is, formally, the solution of an ODE, and so, is a continuous process. This problem has been treated in two directions. In \cite{DNN19}, the products of a discrete jump are spread over the adjacent meshes, according to a distribution kernel. This means that, even if it has jumps of order one, this scale is no longer discrete. This enables convergence in the supremum norm. It has to be noted that this situation is not biologically realistic.  In a second article, \cite{DNN21}, the meshes of the discrete scale are constant : this scale has macro-site and its jumps only depend on these sites. The limit process is a PDE, for the continuous scale, coupled with a jump process, for the discrete scale. This PDMP limit is mathematically interesting and is closer to biology but is not really relevant. 

The goal of the present work is to prove the convergence of the stochastic model under the natural on-site reaction assumption. In order to address the issue of making a discrete process converge to a continuous one, we need to relax the topology on the discrete scale and prove a convergence in negative Sobolev norm. However, it has two costs. The first one is a constraint on the model itself. In order to prove that the limit process is a solution to the limit equation, the rates must be continuous, in the discrete scale, for the Sobolev topology, and so, they must be linear in the discrete scale. The second cost relates to the continuous scale. Indeed, by a duality argument, a weak convergence on the discrete scale must be compensated with a stronger convergence on the continuous scale, with a convergence in a positive Sobolev topology. This suggests that the density assumption $l/\log(N)\to \infty$ will not be strong enough to ensure the convergence. Note in a different context, Blount has also considered convergence in negative Sobolev norms in \cite{Blount91}.

Let us summarise the content of our article. In Section \ref{sec:model}, we present our stochastic model and our main convergence result. We also introduce the functional framework. In section \ref{sec:v}, we prove that the discrete scale is tight. It proceeds from a martingale control in low regularity and does not involve any density assumption. In section \ref{sec:u}, we control the martingale part of the continuous scale when smoothed by the discrete heat kernel. The main tool there is a bound for exponential moments of a function-valued martingale. It is the key step from where our new density assumption arises.  We conclude the proof of convergence in Section \ref{sec:convergence}. It mixes probabilistic convergence arguments and numerical analysis tools in low regularity. Note that we perform our analysis in the case of a single continuous species and a single discrete species but this is is no loss of generality. Also, we work in the one dimensional and periodic case. Extension to other boudary conditions is no problem while higher dimension can be treated at the price of further assumptions, see Remark (see Remark \ref{r}).

\section{The spatial stochastic gene network}
\label{sec:model}
\subsection{The stochastic model}
For $N\geq 1$, we denote by $I_j = [j/N, (j+1)/N[$, for $0\leq j\leq N-1$ a partition of $[0,1[$, with the convention $I_{-1} = I_{N-1}$ and $I_N = I_0$. We are interested in a system of two reactants, denoted as $C$ and $D$, living on the $[0,1]$, with periodic boundary condition. The reactants $C$ are abundant; let $l>1$ be the order of magnitude of $C$ reactants in a cell $I_j$. This parameter may depend on $N$. On the other hand, the reactants $D$ are rare : the number of $D$ reactants in a cell is of order $1$. In total, there are about $lN$ reactants $C$ and $N$ reactants $D$ in the system. The notation $C$ refers to "continuous scale" (or "concentration scale") as, for this species, the relevant information is total number of particle renormalised by the order of magnitude $l$. The notation $D$ refers to the "discrete scale".  For all $0\leq t$ and $0\leq j\leq N-1$, we denote by $lu^N_j(t)$ and $v^N_j(t)$ the number of reactants $C$ and $D$, respectively, in the cell $I_j$, at time $t$.

These two species are subject to a finite number of on-site reactions $(E_r)_{r\in\RR}$, characterised by their jump sizes $(\gamma_r)_{r\in\RR}$ in $\N^*\cup\{-1\}$, and their rates $(\lambda_r)_{r\in\RR}$. There are tow kinds of reactions $\RR = \RR_C\cup\RR_D$, which transform the system in the following way
\[\left\{\begin{array}{lcc}
(u^N_j(t), v^N_j(t)) \longrightarrow (u^N_j(t)+\gamma_r/l, v^N_j(t)) & r\in\RR_C\\
(u^N_j(t), v^N_j(t)) \longrightarrow (u^N_j(t), v^N_j(t)+\gamma_r) & r\in\RR_D\\
\end{array}\right.\]
Each reaction $E_r$ occurs in the cell $I_j$ at rates $\lambda_r(u^N_j, v^n_j)$. 

In the following, we assume that the rates satisfies the following properties.

\begin{ass}\label{ass:rate}
For all $r\in\RR$, we have 
\[\lambda_r (u,v) =\left\{ \begin{array}{lcl}
a_r uv + b_r(u) + d_rv  & \text{if}& r\in\RR_C, \gamma_r>0\\
a_r uv + b_r(u)   & \text{if}& r\in\RR_C, \gamma_r=-1\\
d_rv + b_r(u) & \text{if}& r\in\RR_D, \gamma_r>0\\
d_rv  & \text{if}& r\in\RR_D, \gamma_r=-1\\
\end{array}\right.\]
with $a_r,d_r\geq 0$, and $(b_r)_{r\in\RR}$ are non-negative, $\CC^1$ functions such that $b_r(0)=0$ whenever $\gamma_r=-1$.
\end{ass} 
As already mentionned, $v^N$ is controlled only in negative topologies. Therefore the rates have to be linear with respect to $v$. Assumption \ref{ass:rate} is stronger since we only consider terms of the form $uv$ for reactions in 
$\RR_C$ and $v$ in $\RR_D$. We need to assume this for technical reasons, in particular in the proof of tightness of
$v^N$ performed in section \ref{sec:v}.

We denote by $|\gamma|$ the largest jump size
\[|\gamma|=\max_{r\in\RR}|\gamma_r|.\] 
We also denote by $R_i$, $a_i$, $b_i$, and $d_i$, for $i\in\{C,D\}$, the functions and constants defined as 
\[R_i(u,v) = \sum_{r\in\RR_i} \gamma_r \lambda_r(u,v)= a_iuv+b_i(u)+d_iv, \quad\forall (u,v)\in\R^2.\]

Furthermore, we add a confining field assumption on $R_C$
\begin{ass}\label{ass:RC}
There exists $M>0$ such that for all $u>M$, $v\geq 0$, $R_C(u,v)<0$.
\end{ass}

This assumption means that $a_C<0$, and $b_C$ is negative outside a compact. Together with the on-site reactions, the $C$ species are subject to diffusion. This means that, the system undergoes the following transitions, both with rate $lN^2u^N_j$,
\[\left\{\begin{aligned}
(u^N_{j-1}(t), u^N_j(t)) &\longrightarrow (u^N_{j-1}(t)+1/l, u^N_j(t)-1/l)\\
(u^N_j(t), u^N_{j+1}(t)) &\longrightarrow (u^N_{j}(t)-1/l, u^N_{j+1}(t)+1/l)\\
\end{aligned}\right.\]
The $D$ species are not subject to diffusion and only evolve through on-site reactions. 

Let us define a sequence of stochastic processes associated to these transitions. Let $H^N$ be the space of step function on $[0,1[$, $\R$-valued, constant on the intervals $(I_j)_{0\leq j\leq N-1}$. Let $(P_{r,j})$, $(P_{j^+})$, and $(P_{j^-})$ denote independent Poisson processes with intensity $1$. We define a jump process $(u^N,v^N) : \R_+\times[0,1]\to H_N\times H_N$ as
\[u^N(t,x) = \sum_{j=0}^{N-1}\ind_{I_j}(x)u^N_j,\quad v^N(t,x) = \sum_{j=0}^{N-1}\ind_{I_j}(x)v^N_j,\]
such that
\begin{equation}\label{eq:uNvN}
\left\{\begin{aligned}
u^N(t,x) 
=& u^N(0,x) + \sum_{j=0}^{N-1}\ind_{I_j}(x)\sum_{r\in\RR_C}\frac{\gamma_r}{l}P_{r,j}\left(l\int_0^t \lambda_r(u^N_j(s),v^N_j(s))\, ds\right)\\
& + \sum_{j=0}^{N-1}\frac{\ind_{I_{j+1}}(x)-\ind_{I_j}(x)}{l} P_{j^+}\left(lN^2\int_0^tu^N_j(s)\,ds\right)\\
& + \sum_{j=0}^{N-1}\frac{\ind_{I_{j-1}}(x)-\ind_{I_j}(x)}{l} P_{j^-}\left(lN^2\int_0^tu^N_j(s)\,ds\right)\\
v^N(t,x) =& v^N(0,x) +  \sum_{j=0}^{N-1}\ind_{I_j}(x)\sum_{r\in\RR_D}\gamma_r P_{r,j}\left(\int_0^t \lambda_r(u^N_j(s),v^N_j(s))\, ds\right)\\
\end{aligned}\right.
\end{equation}

In order for this process to be well-defined, we need to prove that the jump rates do not explode. With our assumptions, it is sufficient to show that the process is bounded. Yet, we do not have this information a priori. There are two ways to deal with this issue. The first one is to add a cemetery state and to define the process as stationary after explosion. The second one, consists in truncating the jumps after a stopping time, and in working with this auxiliary process. This is the approach we use. The well-posedness of our original process without truncation will be a by-product of the convergence of the truncated process to a bounded limit.

In order to study these processes, we  emphasise their martingale parts. Indeed, if $h$ is an adapted process, and $P$ is a Poisson process, the process
\[t\mapsto P\left(\int_0^t h_s\, ds \right) -\int_0^t h_s\, ds \] 
is a local martingale. So, we have the following decomposition 
\[\left\{ \begin{aligned}
u^N(t,x) &= u^N(0,x) + Z^N_C(t,x) + \int_0^t R_C(u^N(s,x), v^N(s,x))\, ds + \int_0^t \Delta_N u^N(s,x)\, ds\\
v^N(t,x) &= v^N(0,x) + Z^N_D(t,x) + \int_0^t R_D(u^N(s,x), v^N(s,x))\, ds\\
\end{aligned}\right.\]
where $Z^N_i$ are $H_N$-valued local martingales, and $\Delta_N$ denotes the discrete Laplace operator on $H_N$
\begin{equation}\label{eq:DeltaN}
\Delta_N f(x) = N^2\left(f(x+1/N) - 2f(x) + f(x-1/N)\right),\quad\forall f\in H_N.
\end{equation}

The goal of this article is to prove the convergence, for an appropriate topology, of the sequence $(u^N,v^N)$ to the solution of the coupled equations
\begin{equation}\label{eq:uv}
\left\{ \begin{aligned}
\partial_t u &= \Delta u + R_C(u,v),\\
\partial_t v &= R_D(u,v).\\
\end{aligned}\right.
\end{equation}
These are supplemented by initial conditions and to periodic boundary conditions. 

\subsection{Classical and discrete Sobolev spaces}
In order to prove the convergence of $(u^N,v^N)$, we introduce a Sobolev structure on $H_N$. Let us explain its construction and its basics properties. For more details on discrete Sobolev norms, see \cite{Blountphd}.

Let us denote by $L^2$ the space of square integrable functions on $[0,1]$, and $\langle\cdot,\cdot\rangle$ its canonical inner product.  For $m\in\N^*$, we define the functions
\begin{align*}
\varphi_0 : &x\in[0,1]\mapsto 1\\
\varphi_m : &x\in[0,1]\mapsto \sqrt{2}\cos(2\pi m x)\\
\psi_m : &x\in[0,1]\mapsto \sqrt{2}\sin(2\pi m x)\\
\end{align*}

The family $\left\{\varphi_m,\psi_m\right\}$ is a complete orthonormal family in $L^2$ of eigenfunction for the Laplace operator $\Delta$, with periodic boundary conditions. Let us denote $\lambda_m = (2\pi m)^2$, then we have $\Delta e_m = -\lambda_m e_m$, for $e_m\in\left\{\varphi_m,\psi_m\right\}$. 

For $0<\alpha$, we define the Sobolev spaces $H^\alpha$ as 
\[H^\alpha = \left\{f\in L^2 / \|f\|_{H^\alpha}^2 =\sum_{0\le m} (1+\lambda_m)^\alpha\left(\langle f, \varphi_m\rangle^2  + \langle f, \psi_m\rangle^2\right) <\infty\right\},\]
where we set $\psi_0 = 0$. Its dual space, denoted $H^{-\alpha}$, is the completion of $L^2$ for the norm
\[\|f\|^2_{H^{-\alpha}} = \sum_{0\le m} (1+\lambda_m)^{-\alpha}\left(\langle f, \varphi_m\rangle^2  + \langle f, \psi_m\rangle^2\right).\]
We denote by $(\cdot,\cdot)$ the duality product. The semigroup generated by $\Delta$ on $L^2$, is denoted by $e^{\Delta t}$. It is can be extended to $H^\alpha$ for $\alpha\in\R$ and is a contraction.

Similarly to the classical construction, the discrete Sobolev structures on $H_N$ relies on the eigenfunctions of a discrete Laplace operator. In the following, assume, for simplicity, that $N$ is an odd integer.  We define the orthogonal projection $P_N : L^2 \to H_N$ as 
\[P_Nf(x) = N \int_{j/N}^{(j+1)/N} f(y) dy,\quad \forall x\in I_j.\] 
The space $H_N$ is a subspace of $L^2$ and inherits its Hilbert structure. We denote by $\|\cdot\|_{H_N^0}$, the restriction to $H_N$ of the $L^2$-norm. Then, we define the discrete gradients $\nabla_N^+$ and $\nabla_N^-$, on $H_N$ as 
\[\nabla_N^\pm f(x) = N\left( f(x\pm 1/N) - f(x)\right),\quad\forall f\in H_N.\]
Then, the discrete Laplacian $\Delta_N$, defined in Equation \eqref{eq:DeltaN}, satisfies 
\[\Delta_N f = -\nabla_N^+\nabla_N^- f = -\nabla_N^-\nabla_N^+ f.\]
For all $1\le m\le (N-1)/2$, we define the functions
\begin{align*}
\varphi_{0,N}(x) =& 1\\
\varphi_{m,N}(x) =& \sqrt{2}\cos(2\pi m j/N)\\
\psi_{m,N}(x) =& \sqrt{2}\sin(2\pi m j/N)\\
\end{align*}
for $x\in I_j$ and $0\leq j\le N-1$. We also set $\varphi_{0,N}=0$. The system $\left\{\varphi_{m,N}, \psi_{m,N}\right\}$ is an orthonormal basis of $H_N$ for the norm $H_N^0$. Moreover, these functions are eigenfunctions of $\Delta_N$. Indeed, we have $\Delta_N e_{m,N} = -\lambda_{m,N} e_{m,N}$, for all $e_{m,N}\in \left\{\varphi_{m,N}, \psi_{m,N}\right\}$ and
\[\lambda_{m,N}= 2N^2\left(1-\cos(2\pi m/N)\right).\]

Then, for all $\alpha\in\R$, we define the norm $H_N^\alpha$ as 
\[\|f\|_{H_N^\alpha}^2 =\sum_{m=0}^{(N-1)/2} (1+\lambda_{m,N})^\alpha\left(\langle f, \varphi_{m,N}\rangle^2  + \langle f, \psi_{m,N}\rangle^2\right) .\]
Let us remark that this definition is consistent for $\alpha=0$.  Let us notice that in the discrete case, the duality bracket coincides with the scalar product. The following lemma from \cite{Blount91} summarises some useful properties for our study.

\begin{lemme}[{\cite[Lemma 2.9]{Blount91}}]
\begin{enumerate}[(a)]
\item For all $f,g\in H_N$, $\langle \nabla_N^+f, g\rangle = \langle f,\nabla_N^- g\rangle$.
\item For $e_{m,N}\in \left\{\varphi_{m,N}, \psi_{m,N}\right\}$, $\|\nabla_N^\pm e_{m,N}\|_\infty\leq 2\sqrt{2}\pi m$.
\item There exists $c>0$ such that for all $1\leq m\leq (N-1)/2$
\[c^{-1}\lambda_m \leq \lambda_{m,N}\leq c\lambda_m.\]
\item For all $0\leq \alpha$, there exists $c>0$ such that for all $f\in H_N$
\[c^{-1}\|f\|_{H^{-\alpha}}\leq \|f\|_{H_N^{-\alpha}}\leq c\|f\|_{H^{-\alpha}}.\]
\end{enumerate}
\end{lemme} 
Let us emphasise that the constants, appearing in this lemma, are independent of $N$. The projections $P_N$ and $id -P_N$ are continuous for the Sobolev norm, and for all $\gamma\in [0,1]$, we have
\begin{equation}\label{truc1}
\|P_Nv\|_{H^\gamma_N} \le C\|v\|_{H^\gamma}
\end{equation}
and 
\begin{equation}\label{truc2}
\|(I-P_N)v\|_ {L^2}\le C/N \|v\|_{H^1}
\end{equation}

The functions $\ind_{I_j}$ are of high interest for two reasons. Firstly, every jumps of $\bu^N$ and $\bv^N$ are a combination of them. Secondly, the renormalised $(N\ind_{I_j})_N$ approximate the Dirac measure. That is why we need to control their Sobolev norm. 
\begin{lemme}
For all $-1/2<\gamma$, their exists a constants $C_\gamma>0$ such that
\[C_\gamma^{-1} N^{\gamma-1/2}\leq \left\|\ind_{I_j}\right\|_{H_N^\gamma} \leq C_\gamma N^{\gamma-1/2}.\]
\end{lemme}
\begin{proof}
For all $N\geq 1$, $0\leq j\leq N-1$, and $ 0\leq m\leq (N-1)/2$, we have
\[\langle \ind_{I_j},\varphi_{m,N}\rangle =\frac{\sqrt{2}}{N}\cos(2\pi{j} m/N), \langle \ind_{I_j},\psi_{m,N}\rangle =\frac{\sqrt{2}}{N}\sin(2\pi {j}m/N),\]
and $\langle \ind_{I_j},\varphi_{0,N}\rangle =1/N$.
It follows that 
\begin{align*}
\left\|\ind_{I_j}\right\|_{H_N^\gamma}^2
=& \sum_{m=0}^{(N-1)/2} (1+\lambda_{m,N})^\gamma\left(\langle \ind_{I_j},\varphi_{m,N}\rangle^2 + \langle \ind_{I_j},\psi_{m,N}\rangle^2\right)\\
&=\frac{1}{N^2} +\frac{2}{N^2}\sum_{m=1}^{(N-1)/2} (1+\lambda_{m,N})^\gamma\\
=& \frac{1}{N^2} +2N^{2\gamma-2}\sum_{m=1}^{(N-1)/2}(1/N^2 + \sin^2(\pi m/N))^\gamma\\
\end{align*}
Then, for all $N\geq 1$ and $1\leq m\leq (N-1)/2$, we have 
\[\frac{\pi^2}{\pi^2+1}\leq \frac{\sin^2(\pi m/N)}{1/N^2+\sin^2(\pi m/N)}\leq \frac{1}{2}.\]
For $\gamma>-1/2$, we have the following integral comparison
\[\frac{1}{N}\sum_{m=1}^{(N-1)/2}\sin^{2\gamma}(\pi m/N)\sim \int_0^{1/2}\sin^{2\gamma}(x)\, dx.\]
This ends the proof.
\end{proof}
In particular, $\sqrt{N}\ind_{I_j}$ are an orthonormal basis of $H_N$. Now, we give the discrete equivalent of two important results for classical Sobolev spaces : the regularisation of the heat semigroup and the product rules. Let us denote $T_N$ the semigroup generated by $\Delta_N$. It is a contraction for all norm $H_N^\alpha$ and for the norm $L^\infty$.  Moreover, it regularises functions.

\begin{lemme}\label{prop:regular}
Let $\alpha\leq \beta$, then for all $0<T$, there exists $c>0$, such that for all $1\leq N$, and $\varphi\in H_N$
\[\|T_N(t)\varphi\|_{H_N^\beta}\leq c t^{-(\beta-\alpha)/2} \|\varphi\|_{H_N^\alpha},\quad t\in [0,T].\]
\end{lemme}
\begin{proof}
We have
\begin{align*}
\|T_N(t)\varphi\|_{H_N^\beta}^2
&\leq \langle T_N(t)\varphi ,\varphi_{0,N}\rangle^2+ c\sum_{m=1}^{(N-1)/2} \lambda_{m,N}^\beta\left(\langle T_N(t)\varphi ,\varphi_{m,N}\rangle^2+\langle T_N(t)\varphi ,\psi_{m,N}\rangle^2\right)\\
&= \langle\varphi ,\varphi_{0,N}\rangle^2+ c\sum_{m=1}^{(N-1)/2} \lambda_{m,N}^\alpha\lambda_{m,N}^{\beta-\alpha}e^{-2\lambda_{m,N}t}\left(\langle \varphi ,\varphi_{m,N}\rangle^2+\langle \varphi ,\psi_{m,N}\rangle^2\right)\\
&=  c\left(1+t^{-(\beta-\alpha)}\right) \sum_{m=0}^{(N-1)/2} \lambda_{m,N}^\alpha(\lambda_{m,N}t)^{\beta-\alpha}e^{-2\lambda_{m,N}t}\left(\langle \varphi ,\varphi_{m,N}\rangle^2+\langle \varphi ,\psi_{m,N}\rangle^2\right)\\
\end{align*}
Then, for all $0\leq\gamma$ the function $x\in\R_+\mapsto x^\gamma e^{-2x}$ is upper bounded by $c=(\gamma/2e)^\gamma$. Hence, we have
\[\|T_N(t)\varphi\|_{H_N^\beta}^2 \leq c\left(1+t^{-(\beta-\alpha)}\right)\|\varphi\|_{H_N^\alpha}^2.\]
\end{proof}

We conclude with the product rules in $H_N$. We give simple proofs of these results in Appendix \ref{sec:prod} for the sake of completeness but they could be deduced from product rules in discrete Besov spaces (see instance \cite{MP19}, \cite{GH21}). There are two kind of rules : the direct one, for non-negative Sobolev spaces, and their dual, for negative indices.

\begin{theorem}\label{prop:prod+}
Let $0\leq \alpha, \beta, \gamma\leq 1$, satisfying one of the following assumptions
\begin{enumerate}
\item[(i)] If $1/2<\alpha\vee\beta\leq 1$, and $\gamma< \alpha\wedge\beta$, 
\item[(ii)] If $\alpha\vee\beta\leq 1/2$, and $\gamma < \alpha+\beta-1/2$, 
\end{enumerate}
Then, there exists $C>0$ such that for all $u,v\in H_N$,
$\|uv\|_{H_N^\gamma} \leq C  \|u\|_{H_N^\alpha}\|v\|_{H_N^\beta}$.
\end{theorem}

\begin{coro}\label{prop:prod-}
Let $0\leq \alpha, \beta,\gamma\leq 1$, satisfying one of the following assumptions
\begin{enumerate}[(i)]
\item $\beta\vee\gamma> 1/2$, and $\alpha< \beta\wedge\gamma$,
\item$ 0\leq \beta\vee\gamma\leq 1/2$, and $-\gamma < \beta-\alpha-1/2$, 
\end{enumerate}
Then, there exists $C>0$ such that for all $u,v\in H_N$,
$\|uv\|_{H_N^{-\gamma}} \leq C  \|u\|_{H_N^{-\alpha}}\|v\|_{H_N^\beta}$.
\end{coro}

Once again, we point out that the constants do not depend on $N$. In fact, for the classical Sobolev spaces, sharper product rules hold but we do not need these.

\subsection{The main result}

Let us fix a time horizon $T>0$.  Let $0<\alpha<\beta<1/2$. The goal of this article is to prove the convergence of $(u^N, v^N)$ to this unique solution of Equation \eqref{eq:uv}, in $\CC\left([0,T], (L^\infty\cap H^\beta)\times H^{-\alpha}\right)$, in the sense that 
\[\sup_{0\leq t\leq T}\|u^N(t)-u(t)\|_\infty,\quad  \sup_{0\leq t\leq T}\|u^N(t)-P_N(u)(t)\|_{H_N^\beta},\quad\text{and}\quad \sup_{0\leq t\leq T}\|v^N(t)-v(t)\|_{H_N^{-\alpha}}\]
converge to $0$. Firstly, let us make some assumptions on the initial distributions.

\begin{ass}\label{ass:u0v0}
\begin{enumerate}[(i)]
\item $(u^N(0), v^N(0))$ are uniformly bounded in $L^\infty$ : for all $N\geq 1$
\[\|(u^N(0), v^N(0)\|_\infty<M.\]
\item There exists $\tilde{\beta}>\beta$ such that $u(0)\in \CC^0\cap H^{\tilde{\beta}}$ and $u^N(0)$ converges to $u(0)$ in $L^\infty\cap H_N^{\tilde{\beta}}$.
\item $v^N(0)$ converges to $v(0)$ in $H^{-\alpha}$.
\end{enumerate}
\end{ass}

Under this assumption on $(u(0), v(0)$, we show that Equation \eqref{eq:uv} admits a unique solution.
\begin{lemme}\label{prop:unic}
For all $(u_0,v_0)\in (L^\infty\cap H^\beta) \times H^{-\alpha}$, there exists only one solution to the system \eqref{eq:uv}, with initial condition $u(0) = u_0$ and $v(0) = v_0$, in $\CC\left([0,T], (L^\infty\cap H^\beta)\times H^{-\alpha}\right)$. Moreover, it satisfies 
\[\sup_{0\leq t\leq T}\|u(t)\|_\infty\leq M.\]
\end{lemme}
The proof is left in Appendix \ref{sec:existunic}

Our main result states that under an appropriate scaling property for $l$, the stochastic model $(u^N, v^N)$ converges to the solution of Equation \eqref{eq:uv}.

\begin{theorem}\label{prop:main}
Assume that $lN^{-2\beta}/\log(N)$ converges to $+\infty$. Then under our assumptions, the process $(u^N, v^N)$ converges in probability to the unique solution of the system \eqref{eq:uv}, for the uniform topology on $\DD\left([0,T], (L^\infty\cap H_N^\beta)\times H^{-\alpha}\right)$.
\end{theorem} 

Let us notice that the assumption on the scaling of $l$ suggests that we should think $\beta$ as small as possible, so that $l$ diverges significantly slower than the usual parabolic scaling $l/N^2$, and not much quicker than the scaling $l/\log(N)$  required for the one-scale gene network from \cite{Blount92}.

Let us remark that in this one-dimensional case, it is possible to choose $\alpha=\beta$, even for arbitrarily small $\beta$, at the cost of controlling products in low regularity. Yet, this argument does not hold in larger spatial dimension $d>1$, for small $\beta$. 

An interesting feature of ou result is that it shows that at the limit, the concentration of the continuous  reactant satisfies a reaction diffusion equation with memory. Indeed, the equation on the discrete scale can be resolved explicitly, in terms of $u$ 
\[v(t) = e^{d_Dt}v(0) + \int_0^te^{d_D(t-s)}b_D(u(s))\, ds,\quad t\in[0,T].\]
So $u$ is solution to the equation
\[\partial_t u(t) = \Delta u(t) +b_C(u(t)) +(a_Cu(t)+d_C)\left(e^{d_Dt}v(0)+\int_0^te^{d_D(t-s)}b_D(u(s))\, ds\right).\]
This memory effect has been observed in biological experiments.

\subsection{A truncation argument}
The existence of $(u^N, v^N)$ is easy to prove as long as it stays bounded. However,  we do not have this information, a priori, so we introduce an auxiliary processes $(\bu^N, \bv^N)_{N\geq 1}$, with stopped jumps. Let us define the sequence of stopping times
\[\tau_N = \inf\left\{t\geq 0/ \|u^N(t)\|_{\infty}>M+1\, \text{or}\,\|v^N(t)\|_{H_N^{-\alpha}}>M+1\right\},\]
where $M$ satisfies Assumption \ref{ass:RC}, and $M \geq \sup_{0\leq t\leq T}\|v(t)\|_{H^{-\alpha}}$, where $(u,v)$ is the unique solution of Equation \eqref{eq:uv} with initial data $(u(0), v(0))$. We define the auxiliary processes as, 
\begin{equation}\label{eq:buNbvN}
\left\{ \begin{aligned} 
\bu^N(t) &= u^N_0 + \int_0^t\Delta_N \bu^N(s) + R_C(\bu(s), \bv^N(s))\, ds + Z_c^N(t\wedge\tau_N)\\
\bv^N(t) &= v^N_0 + \int_0^t R_D(\bu^N(s), \bv^N(s))\, ds + Z^N_D(t\wedge\tau_N)\\ 
\end{aligned}\right.
\end{equation}

As a consequence of Assumption \ref{ass:RC}, and the maximum principle, the process $\bu^N$ is bounded in $\DD([0,T],\L^\infty)$.

\begin{prop}\label{prop:ppmax}
Let $l\geq|\gamma|$. Assume that $\|u^N(0)\|_{\infty}<M$ a.s, then, under Assumptions \ref{ass:rate} and \ref{ass:RC}, we have 
\begin{enumerate}[(i)]
\item For all $0\leq t\leq \tau_N$, $\bu^N(t)$ and $\bv^N(t)$ are non negative.
\item $\sup_{0\leq t\leq T}\|\bu^N_t\|_\infty<M+2$.
\item For all $N\geq1$, there exists $C_N>0$ such that $\sup_{0\leq t\leq T}\|\bv^N_t\|_\infty<C_N$.
\end{enumerate}
\end{prop}
\begin{proof}
Assumption \ref{ass:rate} ensures that whenever $\bu^N$ (respectively $\bv^N$) vanishes, then for all $r\in\RR_C$ (respectively $\RR_D$) such that $\gamma_r=-1$, $\lambda_r(\bu^N, \bv^N)=0$. It ensures that $\bu^N$ and $\bv^N$ are non-negative, up to time $\tau_N$.
Let $t_0\leq T$ and $0\leq j_0\leq N-1$ such that $\bu^N_{j_0}(t_0) = \sup_{0\leq t\leq T}\|\bu^N(t)\|_\infty$. Assume that $\bu^N_{j_0}(t_0)>M$ then $\tau_N< t_0\leq T$ and $R_C(\bu^N_{j_0}(t_0), \bv_{j_0}(t_0))<0$. Then $\bu^N_{j_0}$ is differentiable at $t_0$ and we have $\partial_t \bu^N_{j_0}(t_0) \ge 0$.  It follows that $\Delta_N\bu^N(t_0,j_0/N)>0$ which refutes the maximality of $\bu^N_{j_0}(t_0)$. This proves (ii).
Concerning $\bv^N$, let us remark that there exists $c_N>0$ such that for all $N$ and $f\in H_N$
\[\|f\|_\infty\leq c_N\|f\|_{H^{-\alpha}},\]
so for $0\leq t\leq \tau_N$, we have $\|\bv^N(t)\|_\infty\leq c_N\tilde{M}$. Now for $t\geq \tau_N$, we have 
\[\bv^N(t) = \bv(\tau_N) + \int_{\tau_N}^t (d_D\bv^N(s) + b_D(\bu^N(s)))\, ds,\]
So, we obtain an upper bound thanks to Gr\"onwall's inequality.
\end{proof}
Let us emphasise that the $L^\infty$ bound on $\bv^N$ is not uniform on $N$. In the following, we prove that the auxiliary processes $(\bu^N, \bv^N)$ converge to $(u,v)$ and the conclude on the convergence of $(u^N, v^N)$ as a consequence of the boundedness of the limit $(u,v)$.

\section{Tightness of the discrete scale}
\label{sec:v}
The goal of this section is to prove that the discrete scale process $(\bv^N)_N$ is $\CC$-tight for the uniform topology on $\DD([0,T],H^{-\alpha})$. To that extend, we prove that the jump part  vanishes in $\DD([0,T],H^{-\tilde{\alpha}})$ for all $0<\tilde{\alpha}\le 1/2$.

Firstly, we prove a uniform bound on the  $|v_j^N(t)|$.
\begin{lemme}
Assume that $\|(u^N(0), v^N(0))\|_{\infty}<M$ a.s, then the sequence $(\bv^N)_N$ is uniformly bounded : there exists $C>0$ such that for all $N,j$
\[\E\left[\sup_{0\leq t\leq T} |\bv^N_j(t)|\right] < C.\]
\end{lemme}

\begin{proof}
For all $N,j$, and for all $0\leq \bar{t}\leq \tau_N$, we have
\begin{align*}
\bv_j^N(\bar{t}) 
&= v_j^N(0) +\sum_{r\in\RR_D} \gamma_r \PP_{r,j}\left(\int_0^{\bar{t}} \lambda_r(\bu_j^N(s),\bv_j^N(s))\, ds\right)\\
\sup_{0\leq t\leq {\bar{t}}} |\bv^N_j(t)| 
&\leq M +\sum_{r\in\RR_D} |\gamma_r| \PP_{j,r}\left(\int_0^{\bar{t}} \lambda_r(\bu_j^N(s),\bv_j^N(s))\, ds\right)\\
\E\left[\sup_{0\leq t\leq {\bar{t}}} |\bv^N_j(t)|\right] 
&\leq M + C\sum_{r\in\RR_D}  \int_0^{\bar{t}} \E\left[\lambda_r(\bu_j^N(s),\bv_j^N(s))\right]\, ds \\
&\leq C + C \int_0^{\bar{t}}\E\left[\sup_{0\leq t\leq s} |\bv^N_j(t)|\right]\, ds.\\
\end{align*}
With Gr\"onwall's inequality, we obtain
\[\E\left[\sup_{0\leq t\leq \tau_N} |\bv^N_j(t)|\right] < C.\]
To conclude, for $\tau_N\leq t\leq T$, the process $\bv^N$ is solution to a linear ODE with bounded coefficients. This ends the proof.
\end{proof}

Let us notice that this lemma does not give any uniform bound on $\|v^N(t)\|_\infty$. From this bound, we obtain a uniform control on the martingale part $Z^N_D$. Let us denote
\[ M_j^N(t) = \sum_{r\in\RR_D} \gamma_r \left[\PP_{r,j}\left(\int_0^t \lambda_r(\bu_j^N(s),\bv_j^N(s))\, ds\right) - \int_0^t \lambda_r(\bu_j^N(s),\bv_j^N(s))\, ds\right],\]
the martingale part of $v_j^N$. Then, we have
\[Z_D^N(t,x) = \sum_{j=0}^{N-1} M_j^N(t)\ind_{I_j}(x).\]

\begin{lemme}
The local martingales $(M^N_j)_{j}$ are uniformly bounded martingales: there exists $\tilde{C}>0$ such that for all $N,j$
\[\E\left[\sup_{0\leq t\leq T} \left|M^N_j(t\wedge\tau_N)\right|^2\right] < \tilde{C}.\]
\end{lemme}

\begin{proof}
For all $N,j$, and for all $0\leq t\leq \tau_N$, we have
\[\left|M^N_j(t)\right|^2 \leq C\sum_{r\in\RR_D}\left(\PP_{r,j}\left(\int_0^t \lambda_r(\bu_j^N(s),\bv_j^N(s))\, ds\right) - \int_0^t \lambda_r(\bu_j^N(s),\bv_j^N(s))\, ds\right)^2.\]
Then, we have
\[\E\left[\left|M^N_j(t)\right|^2\right]\leq C\sum_{r\in\RR_D}\E\left[\int_0^t \lambda_r(\bu_j^N(s),\bv_j^N(s))\, ds\right]\leq C.\]
This proves that $M^N_j(t)$ is $\L^1$ for all $t$, and so, $(M^N_j)_{j,N}$ are true martingales. Lastly, we conclude the proof by Doob's inequality.
\end{proof}

Then, we prove our first main result: the vanishing of the discrete scale jumps. 

\begin{theorem}\label{prop:saut_disc}
Assume that $\|(u^N(0), v^N(0))\|_{\infty}<M$ a.s. Then, $Z_D^N$ converges to $0$, in $\L^1$, for the uniform topology on $\DD([0,T],H^{-\tilde{\alpha}})$. Moreover, there exists $C(T)>0$ such that
\[\E\left[\sup_{0\leq t\leq T} \|Z_D^N(t\wedge\tau_N)\|^2_{H^{-\tilde{\alpha}}}\right] \leq C N^{-2\alpha}.\]
\end{theorem}

\begin{proof}
For all $\varphi\in H_N$, we have
\[\langle Z_D^N(t), \varphi\rangle  = \sum_{j=0}^{N-1} M_j(t)\langle \ind_{I_j}, \varphi\rangle.\]
Then, we have
\begin{align*}
\|Z^N_D(t)\|^2_{H_N^{-\tilde{\alpha}}} 
=& \sum_{m=0}^{N-1}(1+\lambda_{m,N})^{-\tilde{\alpha}}\left(\langle Z_D^N(t,\cdot), \varphi_{m,N}\rangle^2 +\langle Z_D^N(t,\cdot), \psi_{m,N}\rangle^2\right)\\
=&\sum_{m=0}^{N-1} (1+\lambda_{m,N})^{-\tilde{\alpha}}\left[\left(\sum_{j=0}^{N-1}M_j^N(t)\langle \ind_{I_j}, \varphi_{m,N}\rangle\right)^2+ \left(\sum_{j=0}^{N-1}M_j^N(t)\langle \ind_{I_j}, \psi_{m,N}\rangle\right)^2\right]\\
\end{align*}
By independence of the Poisson processes and thanks to the previous lemma, there exists $C>0$ such that for all $t\in[0,T]$
\[\E\left[ M_j^N(t\wedge\tau_N)M_l^N(t\wedge\tau_N)\right]\leq C \delta_{j,l}.\]
It follows that 
\[\E\left[\|Z^N_D(t\wedge\tau_N)\|^2_{H_N^{-\alpha}} \right]\leq C\sum_{m=0}^{(N-1)/2}(1+\lambda_{m,N})^{-\tilde{\alpha}} \sum_{j=0}^{N-1}\left(\langle \ind_{I_j}, \varphi_{m,N}\rangle^2+\langle \ind_{I_j}, \psi_{m,N}\rangle^2\right).\]
Then, we use that $(\sqrt{N}\ind_{I_j})_{0\leq j\leq N-1}$ is an orthonormal basis of $H_N$. It follows that
\[\E\left[\|Z^N_D(t\wedge\tau_N)\|^2_{H_N^{-\tilde{\alpha}}} \right]\leq \frac{C}{N}\sum_{m=0}^{(N-1)/2}(1+\lambda_{m,N})^{-\tilde{\alpha}}.\]
Lastly, we recall that $\lambda_{0,N}=1$, and, for $m\geq1$, $\lambda_{m,N}$ behaves as $(2\pi m)^2$, independently of $N$. So, the sum over $m\geq1$ is of order $N^{1-2\tilde{\alpha}}$, that is
\[\E\left[\|Z^N_D(t\wedge\tau_N)\|^2_{H_N^{-\tilde{\alpha}}} \right]\leq \frac{C}{N^{2\tilde{\alpha}}}.\]
To conclude, we use Doob's inequality, for the submartingale $\|Z^N_D(t\wedge\tau_N)\|^2_{H_N^{-\tilde{\alpha}}}$ and the equivalence between the discrete and classical Sobolev norm of negative index.
\end{proof}

This theorem has two important consequences. The first one is the uniform boundedness, in $\L^1$, of $\bv^N$, for the uniform topology on  $\DD([0,T],H^{-\tilde{\alpha}})$. 

\begin{coro}[bound in $L^\infty(H_N^{-\tilde{\alpha}})$]\label{prop:disc_bound}
Assume that $\|(u^N(0), v^N(0))\|_{\infty}<M$ a.s, then there exists $C>0$ such that 
\[\E\left[ \sup_{0\leq t\leq T} \|\bv^N_t\|_{H^{-\tilde{\alpha}}}\right]\leq C.\]
\end{coro}

\begin{proof}
For all $0\leq \bar{t}\leq T$, we have
\begin{align*}
\|\bv^N(\bar{t})\|_{H^{-\tilde{\alpha}}} 
&\leq \|v^N(0)\|_{H^{-\tilde{\alpha}}} + \int_0^{\bar{t}} \|R_D(\bu^N(s), \bv^N(s))\|_{H^{-\tilde{\alpha}}} \, ds + \|Z_D^N(\bar{t})\|_{H^{-\tilde{\alpha}}}\\
&\leq M + \int_0^{\bar{t}} |d_D|\|\bv^N(s)\|_{H^{-\tilde{\alpha}}} +\|b_D(\bu^N(s))\|_\infty \, ds + \|Z_D^N(\bar{t})\|_{H^{-\tilde{\alpha}}}\\
\E\left[\sup_{0\leq t\leq {\bar{t}}}\|\bv^N(t)\|_{H^{-\tilde{\alpha}}}\right]
&\leq C + C\int_0^{\bar{t}} \E\left[\sup_{0\leq t\leq s}\|\bv^N(t)\|_{H^{-\tilde{\alpha}}}\right]\, ds\\
\end{align*}
The proof is concluded with Gr\"onwall's inequality.
\end{proof}

The second consequence of Theorem \ref{prop:saut_disc}, is the tightness of $\bv$ for the uniform topology on $\DD([0,T], H^{-\alpha})$.

\begin{theorem}
Assume that $\|(u^N(0), v^N(0))\|_{\infty}<M$ a.s, then $(\bv^N)_N$ is $\CC$-tight for the uniform topology on $\DD\left(([0,T], H^{-\alpha}\right)$.
\end{theorem}

\begin{proof}
According to Theorem \ref{prop:saut_disc}, the sequence $(Z_D^N)_N$ converges in $\L^1$ to $0$ for the uniform topology on $\DD\left(([0,T], H^{-\alpha}\right)$. So, it is tight. Besides, the sequence $(v^N(0))_N$ is bounded by assumptions, so is also tight. Then, it is sufficient to show that the sequence
\[F^N : t\in[0,T]\to\int_0^t R_D(\bu^N(s), \bv^N(s))\, ds\]
is tight. Let us denote
\[\|F^N\|_{Lpz,\tilde{\alpha}} = \sup_{0\leq t\leq T}\|F^N(t)\|_{H^{-\tilde{\alpha}}} + \sup_{0\leq t\leq T}\|(F^N)'(t)\|_{H^{-\tilde{\alpha}}},\]
the Lipschitz norm on $\CC([0,T],H^{-\tilde{\alpha}})$. 
It follows from Corollary \ref{prop:disc_bound} that 
\[\E\left[\|F^N\|_{Lpz,\tilde{\alpha}}\right]\leq C(1+T).\]
Then, for all $0<\tilde{\alpha}<\alpha$, the $H^{-\tilde{\alpha}}$-ball is compact in $H^{-\alpha}$, so from Arzela-Ascoli theorem,  the $\|\cdot\|_{Lpz,\tilde{\alpha}}$-balls of  $\CC\left(([0,T], H^{-\alpha}\right)$ are compact for the uniform topology. So $(F^N)_N$ is tight in $\DD([0,T], H^{-\alpha})$, and so is $(\bv^N)_N$. Moreover, the jumps converges to $0$ in this topology so any limit point is continuous and the sequence is $\CC$-tight.
\end{proof}

\section{Regularised jumps of the continuous scale}
\label{sec:u}
The goal of this section is to prove that, for the continuous scale, the jumps' contributions, in a mild sense, converge to $0$ in probability for the uniform topology on $\DD([0,T],L^\infty\cap H_N^\beta)$.

Let us define \[Y^N_t = \int_0^t T_N (t-s)\, d Z_C^N(s\wedge\tau).\]
Then, $\bar{u}^N$ satisfies the mild formula
\[\bu^N(t) = T_N(t)u^N_0 + Y^N_t + \int_0^t T_N (t-s)R_C\left(\bu^N(s), \bv^N(s)\right)\, ds.\]

The average jump process $Y^N$ is more regular than the jump process $Z^N_C$, and so, we are able to show that $Y^N$ converges in probability to $0$, under weaker assumptions than for $Z^N_C$. Firstly, we need some control on the jumps size of $Z^N_C$. 

\begin{lemme}\label{prop:compensator}
For all $\varphi\in H_N$, the compensator of
\[\sum_{s\leq t}\left(\delta\langle Z^N_C(s\wedge\tau_N),\varphi\rangle\right)^2\]
is 
\[\frac{1}{Nl}\int_0^{t\wedge\tau_N}\langle u^N(s), (\nabla_N^+\varphi)^2+(\nabla_N^-\varphi)^2\rangle + \langle \tilde{R}_C(\bu^N(s), \bv^N(s)), \varphi^2\rangle \, ds,\]
with \[\tilde{R}_C(u,v) = \sum_{r\in\RR_C}\gamma_r^2\lambda_r(u,v).\]
\end{lemme}
\begin{proof}
We denote by $X$ the pure jump process define as
\[X_t = \sum_{s\leq t} \left(\delta\langle Z^N_D(s\wedge\tau_N),\varphi\rangle\right)^2.\]
It has two kinds of jumps. The first ones are the reaction-jumps of size 
$\gamma_r^2\langle \ind_{I_j}, \varphi\rangle^2/l^2$ and with rate $l\lambda_r(\bu^N(s, j/N),\bv^n(s,j/N))$. The second ones are the diffusion-jumps, of size $\langle \ind_{I_{j\pm1}}- \ind_{I_j}, \varphi\rangle^2/l^2$, with rate $lN^2_bu^N(s,j/N)$. Let us remark that, from Proposition \ref{prop:ppmax}, for each $N$ fixed, the rates are uniformly bounded. So according to  \cite[Proposition 2.1]{Kurtz71}, the process 
\[X_t-\int_0^{t\wedge\tau_N} F(s)\, ds\]
is a local martingale, with 
\begin{align*}
F(t) 
=& \sum_{j=0}^{N-1}\sum_{r\in\RR_C}\frac{\gamma_r^2}{l^2}\langle \ind_{I_j}, \varphi\rangle^2 l\lambda_r(\bu^N(s, j/N),\bv^n(s,j/N))\\
& + \sum_{j=0}^{N-1}\frac{1}{l^2}\left(\langle \ind_{I_{j+1}}- \ind_{I_j}, \varphi\rangle^2+ \langle \ind_{I_{j-1}}- \ind_{I_j}, \varphi\rangle^2\right)lN^2\bu^N(t,j/N)\\
=& \frac{1}{l}\sum_{j=0}^{N-1}\langle \ind_{I_j}, \varphi\rangle^2 \tilde{R}_C(\bu^N(t,j/N), \bv(t,j/N))\\
&+ \frac{1}{l}\sum_{j=0}^{N-1} \bu^N(t,j/N)\left((\varphi((j+1)/N)-\varphi(j/N))^2+(\varphi((j-1)/N)-\varphi(j/N))^2\right)\\
= & \frac{1}{Nl}\langle \bu^N(t), \tilde{R}_C(\bu^N(t), \bv^N(t))\rangle + \frac{1}{N^2l} \sum_{j=0}^{N-1} \bu^N(t,j/N)\left((\nabla_N^+\varphi(j/N))^2 + (\nabla_N^-\varphi(j/N))^2\right)
\end{align*} 
This ends the proof.
\end{proof}
Let us denote 
\[\tilde{R}_C(u,v) = \tilde{a}uv + \tilde{b}(u)+\tilde{d}v,\]
with $\tilde{a}, \tilde{d}\geq 0$ and $\tilde{b}$ a non-negative function. We begin with the convergence in the supremum norm. We recall this auxiliary lemma from \cite{Blount92} which links exponential moments of a martingale with its quadratic variation.

\begin{lemme}[\cite{Blount92}]\label{prop:exp_linfty}
Let $\xi$ be a real-valued bounded martingale of finite variation, right-continuous with left limits, defined on $[0,T]$ with $\xi(0)=0$ and such that $|\delta \xi_t|\leq 1$, for $0\leq t\leq T$ and
\[\sum_{s\leq t} \left(\delta\xi_s\right)^2 - \int_0^t g(s)\, ds\]
is a martingale, where $g$ is adapted and $0\leq g\leq h$, for some bounded deterministic function $h$. Then,
\[\E\left[\exp(\xi_T)\right] \leq \exp\left(\frac{3}{2}\int_0^T h(s)\, ds\right).\]
\end{lemme}

The control on the jumps form Lemma \ref{prop:compensator}, and the previous Lemma, allow proving the convergence in $\L^\infty$.

\begin{theorem}\label{prop:saut_cont_infty}
Assume that $\alpha<3/2$, and $l/\log(N)\to+\infty$, then $(Y^N)_N$ converges to $0$ in probability for the uniform topology on $\DD\left(([0,T], L^\infty\right)$. Moreover, there exists $C>0$ such that for all $\varepsilon>0$
\[\P\left(\sup_{0\leq t\leq T}\|Y^N_t\|_{\infty}>\varepsilon\right)\leq C\exp\left(\log(N)-\varepsilon^2l\right).\]
\end{theorem}

\begin{proof}
This proof follows the steps of \cite{Blount92}. Let $\bar{t}\leq T$, $0\leq j\leq N-1$ and $\varphi = N\ind_{I_j}$. We define the process 
\[\xi_t = \left\langle \int_0^t T_N(\bar{t}-s)\, dZ^N_D(s\wedge\tau_N),\varphi \right\rangle.\]
This process is a martingale, and it satifies $\xi_{\bar{t}} = Y^N_{\bar{t}}(j/N)$. By integration by parts, we have 
\[\delta\xi_s = \langle T_N(\bar{t}-s)\delta Z^N_D(s), \varphi\rangle = \langle \delta Z^N_D(s), T_N(\bar{t}-s)\varphi\rangle .\]
For all $0\leq s\leq T$, its jumps satisfy
\begin{align*}
|\delta\xi_s|
&= |\langle T_N(\bar{t}-s)\delta Z^N_C(s\wedge\tau_N), \varphi\rangle|\\
&\leq \|\delta Z^N_C(s\wedge\tau_N) \|_{H^0} \| T_N(\bar{t}-s)\varphi\|_{H^0}\\
&\leq \frac{|\gamma|}{l\sqrt{N}} \|\varphi\|_{H^0}\\
&\leq \frac{|\gamma|}{l}.\\
\end{align*}
In order to apply Lemma \ref{prop:exp_linfty}, we need to control the compensator of $\sum (\delta\xi_s)^2$. 
According to Lemma \ref{prop:compensator}, this compensator is given by $\int g$ with
\begin{align*}
g_s = \frac{1}{Nl}& \left(\left\langle \bu^N(s), \left(\nabla_N^+T_N(\bar{t}-s)\varphi\right)^2+\left(\nabla_N^-T_N(\bar{t}-s)\varphi\right)^2\right\rangle\right.\\ 
&\left.+ \left\langle \tilde{R}_C\left(\bu^N(s), \bv^N(s)\right), \left(T_N(\bar{t}-s)\varphi\right)^2\right\rangle \right).
\end{align*}
Then, for all $0\leq s\leq \bar{t}\wedge\tau_N$,  we have the following bounds, on the diffusion-related term
\begin{align*}
\langle u^N(s), (\nabla_N^\pm T_N(\bar{t}-s)\varphi)^2\rangle
&\leq  M  \|\nabla_N^\pm T_N(\bar{t}-s)\varphi \|_{H_N^0}^2\\
&\leq M \| T_N(\bar{t}-s)\varphi \|_{H_N^1}^2\\
\end{align*}
Moreover, we have
\begin{align*}
\int_0^{\bar{t}}\| T_N(\bar{t}-s)\varphi \|_{H_N^1}^2\, ds
&= \int_0^{\bar{t}}\sum_{m=0}^{(N-1)/2} (1+\lambda_{m,N})e^{-2\lambda_{m,N}(\bar{t}-s)}\left(\langle\varphi,\varphi_{m,N}\rangle^2+\langle\varphi,\psi_{m,N}\rangle^2\right)\, ds\\
&\leq \sum_{m=0}^{(N-1)/2} \left(\langle\varphi,\varphi_{m,N}\rangle^2+\langle\varphi,\psi_{m,N}\rangle^2\right)\\
&\leq \| \varphi \|_{H_N^0}^2\\
&\leq CN.\\
\end{align*}
On the other hand, the reaction-related term rewrites as
\[\langle \tilde{R}_C(\bu^N(s), \bv^N(s)), \left(T_N(\bar{t}-s)\varphi\right)^2\rangle
 = \langle \tilde{a} \bu^N(s)\bv^N(s) + \tilde{b}(\bu^N(s)) + \tilde{d}\bv^N(s), \left(T_N(\bar{t}-s)\varphi\right)^2\rangle.\]
The $\tilde{b}$-related term is bounded by its $H^0$ norm and we have 
\[\langle \tilde{b}(\bu^N(s)),\left(T_N(\bar{t}-s)\varphi\right)^2\rangle \leq \|\tilde{b}\|_\infty\|T_N(\bar{t}-s)\varphi\|^2_{H^0}\leq CN.\]

The other terms are controlled, using the semigroup regularisation Lemma \ref{prop:regular}, and the product rules Theorem \ref{prop:prod+}. Let $0<\varepsilon$ such that $2\alpha+\varepsilon<3$, as $\bu^N$ and $\bv^N$ are non-negative, we have
\begin{align*}
\langle \tilde{a}\bu^N(s)\bv^N(s),\left(T_N(\bar{t}-s)\varphi\right)^2\rangle
&\leq \tilde{a}M\langle\bv^N(s), \left(T_N(\bar{t}-s)\varphi\right)^2\rangle\\
&\leq C\left\|\bv^N(s)\right\|_{H_N^{-\alpha}}\left\|\left(T_N(\bar{t}-s)\varphi\right)^2\right\|_{H_N^{\alpha}}\\
&\leq C\left\|\varphi\right\|_{H_N^{\alpha}}\left\|T_N(\bar{t}-s)\varphi\right\|_{H_N^{\varepsilon/2+1/2}}\\
&\leq CN^{\alpha+1/2}(\bar{t}-s)^{-(2\alpha+1+\varepsilon)/4}\left\|\varphi\right\|_{H_N^{-\alpha}}\\
&\leq CN(\bar{t}-s)^{-(2\alpha+1+\varepsilon)/4}.\\
\end{align*} 
The last term is controlled by the same bound. Therefore, on $[0,\bar{t}\wedge\tau_N]$, $g$ is bounded by a function $h$ such that
\[\int_0^{\bar{t}}h_s\, ds\leq Cl^{-1}.\]
Now, for $\theta\in[0,1]$, we apply Lemma \ref{prop:exp_linfty} to the martingale $|\gamma|^{-1}\theta l\xi$, whose jumps are bounded by $1$. It follows that
\[\E\left[\exp\left(|\gamma|\theta l\xi_{\bar{t}}\right)\right]\leq \exp\left(C\theta^2l\right)
.\]
Therefore, for $\varepsilon>0$, we have 
\begin{align*}
\P\left(Y^N_{\bar{t}}(j/N)>\varepsilon\right)
&= \P\left(\exp\left(|\gamma|\theta l\xi_{\bar{t}}\right)>\exp\left(|\gamma|\theta l\varepsilon\right)\right)\\
&\leq \exp\left(C\theta^2 l -|\gamma|\theta l\varepsilon\right)\\
&\leq \exp\left(-\tilde{C}\varepsilon^2 l)\right).\\
\end{align*}
where the last inequality is obtained with the optimal $\theta=|\gamma|\varepsilon/2C\in[0,1]$. With the same arguments, applied to $-Y^N_{\bar{t}}$, and taking the supremum on $j$, we obtain 
\[\P\left(\|Y^N_{\bar{t}}\|_{\infty}>\varepsilon\right)\leq 2N\exp\left(-\tilde{C}\varepsilon^2 l)\right).\]

To conclude, we need to obtain an inequality for the supremum over $\bar{t}$. We decompose the interval $[0,T]$ as the union of $N^2$ sub-intervals $\II_n = [nT/N^2, (n+1)T/N^2]$. For all $0\leq n\leq N^2-1$, for all $ t\in \II_n$, $Y^N$ satisfies
\[Y^N_t = Y^N_{nT/N^2} + \int_{nT/N^2}^t\Delta_N Y^N_s\, ds + \zeta^n_t,\]
where $\zeta^n_t = Z^N_C(t\wedge\tau_N) - Z^N_C(nT/N^2\wedge\tau_N)$. Then, we have
\[\|Y^N_t\|_\infty \leq \|Y^N_{nT/N^2}\|_{\infty} +4N^2\int_{nT/N^2}^t\|Y^N_s\|_{\infty}\, ds + \|\zeta^n_t\|_{\infty}.\]
By Gr\"onwall's inequality, it follows that
\[\sup_{t\in\II_n}\|Y^N_t\|_{\infty} \leq \left(\|Y^N_{nT/N^2}\|_{\infty} + \sup_{t\in\II_n}\|\zeta^n_t\|_{\infty}\right)e^{4T}.\]
We apply Lemma \ref{prop:exp_linfty} to the martingale $\zeta^n(k/N)=\langle \zeta,\varphi\rangle$. For the jumps' size, we have 
\[|\delta\zeta^n_t(k/N)| = |\delta Z^N_C(t\wedge\tau_N, k/N)|\leq l^{-1}.\]
On the other hand, the compensator of $\sum (\delta\zeta^n(k/N)_s)^2$ is given by 
\[\tilde{g}_s = \frac{1}{Nl}\left(\langle\bu^N(s), (\nabla_N^+\varphi)^2+(\nabla_N^-\varphi)^2\rangle + \langle \tilde{R}_C(\bu^N(s), \bv^N(s)),\varphi^2\rangle \right),\]
with the upper-bounds
\begin{align*}
\langle\bu^N(s), (\nabla_N^+\varphi)^2\rangle
&\leq C\|\varphi\|_{H_N^1}^2\leq CN^3\\
\langle \tilde{a}\bu^N(s)\bv^N(s),\varphi^2\rangle
&\leq C\|\varphi\|_{H_N^\alpha}\|\varphi\|_{H_N^{1/2+\varepsilon}}\leq CN^3\\
\langle \tilde{b}(\bu^N(s)),\varphi^2\rangle
&\leq C\langle \|\varphi\|_{H_N^0}^2\leq CN.\\
\end{align*}
It follows that $\tilde{g}$ is upper-bounded by a function $\tilde{h}$ such that
\[\int_{\II_n}\tilde{h}_s\, ds \leq Cl^{-1}.\]
Therefore, for all $\theta\in[0,1]$, we have
\[\E\left[\exp\left(\theta l\zeta^n((n+1)T/N^2, k/N)\right)\right]\leq \exp\left(C\theta l\right).\]
As $\zeta^n(\cdot,k/N)$ is a martingale, we use Doob's inequality to prove that
\[\P\left(\sup_{t\in\II_n}\zeta^n(t,k/N)>\varepsilon\right)\leq \exp(-\tilde{C}l\varepsilon^2).\]
At the end, we have,
\begin{align*}
\P\left(\sup_{0\leq t\leq T}\|Y^N_t\|_{\infty}>\varepsilon\right)
&\leq \sum_{n=0}^{N^2-1} \P\left(\sup_{t\in\II_n}\|Y^N_t\|_{\infty}>\varepsilon\right)\\
&\leq \sum_{n=0}^{N^2-1}4N\exp\left(-Cl\varepsilon^2\right)\\
&\leq C\exp\left(\log(N)-\varepsilon^2l\right).\\
\end{align*}
\end{proof}

As the sequence $(\bv^N)_{N\geq 1}$ only converges in a negative Sobolev space, we need more than the $L^\infty$ convergence of $(\bu^N)_{N\geq1}$ to prove the convergence of the product $(\bu^N\bv^N)_{N\geq1}$. In order to obtain a convergence result in the $H_N^{\beta}$ topology, we use a functional  version  of Lemma \ref{prop:exp_linfty}, for $H_N$-valued martingales.

\begin{lemme}\label{prop:exp_beta}
Let $\xi$ be a $H_N$-valued bounded martingale of finite variation, right-continuous with left limits, defined on $[0,T]$ with $\xi(0)=0$ and such that $\|\delta \xi_t\|_{H_N^\beta}\leq 1$, for $0\leq t\leq T$ and
\[\sum_{s\leq t} \|\delta\xi_s\|_{H_N^\beta}^2 - \int_0^t g(s)\, ds\]
is a martingale, where $g$ is adapted and $0\leq g\leq h$, for some bounded deterministic function $h$. Then,
\[\E\left[\exp\left(\|\xi_T\|_{H_N^\beta}\right)\right] \leq \exp\left(1+4\int_0^T h(s)\, ds\right).\]
\end{lemme}
\begin{proof}
Let be $f : x\in H_N\mapsto \exp\left(\sqrt{1+\|x\|^2_{H_N^\beta}}\right)$. For all $x,y\in H_N$ such that $\|y\|_{H_N^\beta}\leq 1$, we have
\[f(x+y)\leq \exp\left(1+\|x+y\|_{H_N^\beta}\right)\leq \exp\left(1+\|y\|_{H_N^\beta}+\|x\|_{H_N^\beta}\right)\leq e^2 f(x).\] 
Now, applying Ito formula, we have,
\begin{align*}
f(\xi_t) 
& = f(0) + \int_0^t df(\xi_{s^-})\, d\xi_s + \sum_{s\leq t} f(\xi_s) - f(\xi_{s^-}) - df(\xi_{s^-})(\delta\xi_s)\\
& = e + \int_0^t df(\xi_{s^-})\, d\xi_s + \sum_{s\leq t} \frac{1}{2}d^2f(x_s)(\delta\xi_s,\delta\xi_s)
\end{align*}
with $x_s\in[\xi_{s^-}, \xi_s]$. Moreover, we have
\[d^2f(x) (y,y) \leq \|y\|_{H_N^\beta} f(x).\]
It results that
\[d^2f(x_s)(\delta\xi_s,\delta\xi_s) \leq  f(x_s) \leq  e^2 f(\xi_s).\]
Taking the expectation, we obtain
\[\E[f(\xi_t)] \leq e + 4\int_0^t h_s \E[f(\xi_s)]\, ds.\]
The proof is concluded with Gr\"onwall's inequality.
\end{proof}

Using this lemma, we are able to prove a Sobolev version of Theorem \ref{prop:saut_cont_infty}.

\begin{theorem}\label{prop:saut_cont_beta}
Assume $l/(N^{2\beta}\log(N))\to+\infty$, then $(Y^N)_N$ converges to $0$ in probability for the uniform topology on $\DD\left(([0,T], H_N^{\beta}\right)$, i.e there exists $C>0$ such that for all $\varepsilon>0$
\[\P\left(\sup_{0\leq t\leq T}\|Y^N_t\|_{H_N^\beta}>\varepsilon\right) \leq C\exp\left(\log(N)-\varepsilon^2lN^{-2\beta}\right).\]
\end{theorem}

\begin{proof}
This proof follows the same steps as for the proof of Theorem \ref{prop:saut_cont_infty}. Let $0\leq\bar{t}\leq T$, $0\leq k\leq N-1$. We define the $H_N$-valued martingale 
\[\xi_t =  \int_0^t T_N(\bar{t}-s)\, dZ^N_D(s\wedge\tau_N).\]
This process satisfies $\xi_{\bar{t}} = Y^N_{\bar{t}}$.  

For all $0\leq t\leq \bar{t}$, we have
\[ \delta \xi_t = T_N(\bar{t}-t)\delta Z^N_C(t\wedge\tau_N).\]
Then, on the one hand, we have the following bound on the jumps size
\[\|\delta \xi_t\|_{H_N^\beta} \leq \|\delta Z^N_C(t\wedge\tau_N)\|_{H_N^\beta}\leq cl^{-1} N^{\beta-1/2}.\]

On the other hand, we have
\begin{align*}
\sum_{s\leq t} &\|\delta \xi_t\|_{H_N^\beta}^2\\
&= \sum_{s\leq t}\sum_{m=0}^{(N-1)/2} (1+\lambda_{m,N})^\beta \left(\langle T_N(\bar{t}-s) \delta Z^N_C(s\wedge\tau_N), \varphi_{m,N}\rangle^2 + \langle T_N(\bar{t}-s) \delta Z^N_C(s), \psi_{m,N}\rangle^2\right)\\
&= \sum_{m=0}^{(N-1)/2} (1+\lambda_{m,N})^\beta \sum_{s\leq t}e^{-2(\bar{t}-s)\lambda_{m,N}} \left(\langle  \delta Z^N_C(s\wedge\tau_N), \varphi_{m,N}\rangle^2 + \langle  \delta Z^N_C(s\wedge\tau_N), \psi_{m,N}\rangle^2\right).\\
\end{align*}

Hence, from Lemma \ref{prop:compensator}, the compensator has the form of Lemma \ref{prop:exp_beta} with
\begin{align*}
g_t 
= \frac{1}{Nl}\sum_{m=0}^{(N-1)/2} (1+\lambda_{m,N})^\beta e^{-2(\bar{t}-t)\lambda_{m,N}}&\left( \left\langle \bar{u}^N(t), (\nabla_N^+\varphi_{m,N})^2+(\nabla_N^-\varphi_{m,N})^2\right\rangle \right.\\
&+\left\langle \bar{u}^N(t), (\nabla_N^+\psi_{m,N})^2+(\nabla_N^-\psi_{m,N})^2\right\rangle\\
& \left. + \left\langle \tilde{R}_C(\bu^N(t), \bv^N(t)), \varphi_{m,N}^2 + \psi_{m,N}^2\right\rangle \right).\\
\end{align*}
Then, for all $0\leq s\leq t\wedge\tau_N$,  we have  
\begin{align*}
\langle \bu^N(s), (\nabla_N^+ \varphi_{m,N})^2\rangle
&\leq  \|\bu^N(s)\|_{H^0} \|(\nabla_N^+ \varphi_{m,N})^2\|_{H^0}\\
&\leq  M \|(\nabla_N^+ \varphi_{m,N})\|^2_{\infty}\\
&\leq 2M\pi^2m^2\\
&\leq cM\lambda_{m,N}\\
\end{align*}
This bound also holds replacing $\nabla_N^+$ or $\varphi_{m,N}$ by $\nabla_N^-$ or $\psi_{m,N}$. On the other hand, we have
\begin{align*}
\left\langle \tilde{R}_C(\bu^N(t), \bv^N(t)), \varphi_{m,N}^2 + \psi_{m,N}^2\right\rangle
& = \langle \tilde{a}\bu^N(t)\bv^N(t) + \tilde{b}(\bu^N(t)) + \tilde{d}\bv^N(t), \mathbf{1}\rangle\\
&\leq \tilde{a}\|\bu^N(t)\|_{H_N^\alpha}\|\bv^N(t)\|_{H_N^{-\alpha}} + \|\tilde{b}\|_{\infty,[0,M]} + \tilde{d}\|\bv^N(t)\|_{H_N^{-\alpha}}\\
&\leq C\\
\end{align*}

It follows that $g$ is bounded by the function 
\[h_t = \frac{C}{Nl}\sum_{m=0}^{(N-1)/2} (1+\lambda_{m,N})^{\beta+1} e^{-2(\bar{t}-t)\lambda_{m,N}},\]
which satisfies
\[\int_0^{\bar{t}} h_s\, ds \leq \frac{C}{Nl}\sum_{m=0}^{(N-1)/2} (1+\lambda_{m,N})^{\beta}\leq C\frac{N^{2\beta}}{l},\]
where the constant $C$ does not depend on $N$ nor $\bar{t}$. For $\theta\in[0,1]$, we apply Lemma \ref{prop:exp_beta} to the martingale $\frac{\theta l}{N^{\beta-1/2}}\xi$, whose jumps are bounded by $1$ in $H_N^\beta$-norm. It follows that
\[\E\left[\exp\left(\frac{\theta l}{N^{\beta-1/2}}\|\xi_{\bar{t}}\|_{H_N^\beta}\right)\right]\leq \exp\left(1+C\frac{N^{2\beta}}{l}\frac{\theta^2l^2}{N^{2\beta-1}}\right)=\exp\left(1+C\theta^2 lN\right).\]
Then, for $\varepsilon>0$, we have 
\[\P\left(\|Y^N_{\bar{t}}\|_{H_N^\beta}>\varepsilon\right)\leq \exp\left(1-C\varepsilon^2lN^{-2\beta}\right),\]
by taking the optimal $\theta = \varepsilon N^{-(\beta+1/2)}/(2C)\in[0,1]$. To conclude, we need to obtain an inequality for the supremum over $\bar{t}$. We decompose the interval $[0,T]$ as in the proof of Theorem \ref{prop:saut_cont_infty}. Then, for all $0\leq n\leq N^2-1$, and for all $t\in\II_n$, we have
\[\|Y^N_t\|_{H_N^\beta} \leq \|Y^N_{nT/N^2}\|_{H_N^\beta} +4N^2\int_{nT/N^2}^t\|Y^N_s\|_{H_N^\beta}\, ds + \|\zeta^n_t\|_{H_N^\beta}.\]
By Gr\"onwall's inequality, it follows that
\[\sup_{t\in\II_n}\|Y^N_t\|_{H_N^\beta} \leq \left(\|Y^N_{nT/N^2}\|_{H_N^\beta} + \sup_{t\in\II_n}\|\zeta^n_t\|_{H_N^\beta}\right)e^{4T}.\]
The martingale $\zeta^n$ satisfies $\delta\zeta^n_t = \delta Z^N_C(t\wedge\tau_N)$. It follows that we can apply Lemma \ref{prop:exp_beta} to $\zeta^n$ with 
\[\|\delta\zeta^n_t\|_{H_N^\beta}\leq C\frac{N^{\beta-1/2}}{l},\]
and
\[\tilde{g}_t  \leq \frac{cM}{Nl}\sum_{m=0}^{(N-1)/2}\lambda_{m,N}^{\beta+1}\leq C\frac{N^{2\beta+2}}{l}.\]

Using Doob's inequality for the martingale $\zeta^n$, it follows that, for all $n$
\[\E\left[\exp\left(\frac{\theta l}{N^{\beta-1/2}}\sup_{t\in\II_n}\|\zeta_{t}\|_{H_N^\beta}\right)\right]\leq\exp\left(1+C\theta^2 lN\right),\]
and
\[\P\left(\sup_{t\in\II_n}\|\zeta_{t}\|_{H_N^\beta}>\varepsilon\right) \leq \exp\left(1-C\varepsilon^2lN^{-2\beta}\right).\]
Therefore, we have
\[\P\left(\sup_{t\in[0,T]}\|Y^N_{t}\|_{H_N^\beta}>\varepsilon\right)\leq C\exp\left(\log(N)-\varepsilon^2lN^{-2\beta}\right).\]
This ends the proof.
\end{proof}

Let us notice that this result extends the work of \cite{Blount91}, for positive Sobolev spaces and for non-linear rate functions.  The convergence in $H_N^\beta$ requires a stronger assumption on the behaviour of $l$ as $N$ goes to $\infty$, than the convergence in the supremum norm. However, for $\beta<1/2$, the behaviour of $l$ is still sub-linear, which is far better than the quadratic behaviour $l/N^2\to\infty$ needed to guarantee a direct control of $Z^N_C$ (see \cite{BMM} for example). This suggests that $\beta$ should be thought as being small. 

\section{Convergence}
\label{sec:convergence}
The goal of this section is to prove the convergence of $(\bu^N, \bv^N)_{N\geq 1}$. The core argument is to prove that, along every converging subsequence for $(\bv^N)_{N\geq 1}$, the couple $(\bu^N, \bv^N)_{N\geq 1}$ converges to the solution of the system \eqref{eq:uv}. This is sufficient to prove the convergence. Indeed, it implies that the sequence $(\bv^N)_{N\geq 1}$ has at most one limit point, and as it is tight, it necessarily converges. Hence, the whole sequence $(\bu^N, \bv^N)_{N\geq 1}$ converges to the solution of \eqref{eq:uv}. In order to prove the convergence along subsequence, we proceed in two steps. The first one is probabilistic. We build an auxiliary sequence $(w^N)_{N\geq 1}$ such that $\bu^N-w^N$ converges to $0$ in probability, for the uniform topology on $\DD([0,T], L^\infty\cap H_N^\beta)$. This auxiliary process is the solution of a discrete PDE. Hence, the second step consists in proving the convergence of $(w^N, \bv^N)_{N\geq1}$ to the solution of \eqref{eq:uv}. This second step relies on numerical analysis arguments.

\subsection{The discrete PDE}

Let $v$ be a limit point of $(\bv^N)_{N\geq1}$. Up to taking a subsequence, we may assume that $\bv^N$ converges in distribution to $v$, and using Skorokhod's representation theorem, we may assume that the convergence holds almost surely.  We introduce the auxiliary processes $(w^N)_{N\geq1}$ defined as the solution of the discrete PDE : for all $0\leq t\leq T$
\begin{equation}\label{eq:wN}
w^N(t) = T_N(t)u^N_0 + \int_0^t T_N(t-s)R_C(w^N(s), v^N(s))\, ds.
\end{equation}

Let us remark that, as for $\bu^N$, the maximum principle shows that $w^N$ is bounded in $L^\infty$ and satisfies 
\[\sup_{0\leq t\leq T}\|w^N_t\|_\infty \leq M.\]

The control of the jumps' contribution $Y^N$, from Theorem \ref{prop:saut_cont_infty} and \ref{prop:saut_cont_beta}, implies the convergence of $\bar{u}^N-w^N$ to $0$ for both the Sobolev and the $L^\infty$ topologies.

\begin{prop}\label{prop:uNwN}
Assume that $lN^{-2\beta}/\log(N)\to\infty$ and $\|(u^N(0), v^N(0))\|_{\infty}<M$ a.s, then $\bu^N-w^N$ converges in probability to $0$ for the uniform topology on $\DD([0,T],L^\infty\cap H_N^{\beta})$, i.e for all $\varepsilon>0$, we have
\[\lim_{N\to\infty}\P\left(\sup_{0\leq t\leq T}\|\bar{u}^N(t)-w^N(t)\|_{H_N^\beta}>\varepsilon\right)=\lim_{N\to\infty}\P\left(\sup_{0\leq t\leq T}\|\bar{u}^N(t)-w^N(t)\|_{\infty}>\varepsilon\right)=0.\]
\end{prop}

\begin{proof}
Let $\gamma>1/2$ such that $\gamma+\beta<2$. For all $0\leq t\leq T$, we have
\begin{align*}
\bu^N_t-w^N_t 
&= Y^N_t + \int_0^t T_N(t-s) \left(a(\bu^N_s-w^N_s)\bv^N_s + b_C(\bu^N_s)-b_C(w^N_s)\right)\, ds\\
\|\bu^N_t-w^N_t \|_{H_N^\beta}
&\leq \|Y^N_t\|_{H_N^\beta} +\int_0^t\|T_N(t-s) \left(a(\bu^N_s-w^N_s)\bv^N_s + b_C(\bu^N_s)-b_C(w^N_s)\right)\|_{H_N^\beta}\, ds\\
&\leq \|Y^N_t\|_{H_N^\beta} +C\int_0^t (t-s)^{-(\beta+\gamma)/2}\|(\bu^N_s-w^N_s)\bv^N_s\|_{H_N^{-\gamma}}\, ds \\
&\quad + C\int_0^t (t-s)^{-(\beta+\gamma)/2}+\| b_C(\bu^N_s)-b_C(w^N_s)\|_{H_N^{-\gamma}}\, ds\\
\end{align*}
Then, we have 
\[\|(\bu^N_s-w^N_s)\bv^N_s\|_{H_N^{-\gamma}}\leq c \|\bv^N_s\|_{H_N^{-\alpha}}\|\bu^N_s-w^N_s\|_{H_N^{\beta}}.\]
On the other hand, we have
\[\|b_C(\bu^N_s)-b_C(w^N_s)\|_{H_N^{-\gamma}}\leq \|b_C(\bu^N_s)-b_C(w^N_s)\|_{H_N^{0}}\leq L\|\bu^N_s-w^N_s\|_{H_N^{0}}\leq L\|\bu^N_s-w^N_s\|_{H_N^{\beta}} ,\]
where $L$ is the Lipschitz constant of $b_C$ on $[0,M+1]$.
Hence, we have proved that
\begin{align*}
\|\bu^N_t-w^N_t \|_{H_N^\beta}
&\leq c(1+\sup_{0\leq \sigma\leq T}\|\bv^N_\sigma\|_{H_N^{-\alpha}})\int_0^t(t-s)^{-(\beta+\gamma)/2}\sup_{0\leq \sigma\leq s}\|\bu^N_\sigma-w^N_\sigma \|_{H_N^\beta}\, ds\\
&+\sup_{0\leq \sigma\leq T}\|Y^N_\sigma\|_{H_N^\beta} \\
\end{align*}
Let $p,q> 1$ such that $1/p+1/q=1$ and $q(\beta+\gamma)/2<1$. We have
\begin{align*}
\|\bu^N_t-w^N_t \|_{H_N^\beta}^p
\leq&  C(1+\sup_{0\leq \sigma\leq T}\|\bv^N_\sigma\|_{H_N^{-\alpha}}^p)\left(\int_0^t(t-s)^{-(\beta+\gamma)/2}\sup_{0\leq \sigma\leq s}\|\bu^N_\sigma-w^N_\sigma \|_{H_N^\beta}\, ds\right)^p\\
&+ C\sup_{0\leq \sigma\leq T}\|Y^N_\sigma\|_{H_N^\beta}^p \\
\leq& C(1+\sup_{0\leq \sigma\leq T}\|\bv^N_\sigma\|_{H_N^{-\alpha}}^p)\left(\int_0^t(t-s)^{-q(\beta+\gamma)/2}\, ds\right)^{p/q}\! \int_0^t\sup_{0\leq \sigma\leq s}\|\bu^N_\sigma-w^N_\sigma \|_{H_N^\beta}^p\, ds\\
&+ C\sup_{0\leq \sigma\leq T}\|Y^N_\sigma\|_{H_N^\beta}^p \\
\end{align*}
Then, the singular integral is bounded on $[0,T]$, and, for all $0\leq t\leq T$, we have
\[\sup_{0\leq \sigma\leq t}\|\bu^N_\sigma-w^N_\sigma \|_{H_N^\beta}^p \leq C\sup_{0\leq \sigma\leq T}\|Y^N_\sigma\|_{H_N^\beta}^p +C\left(1+\sup_{0\leq \sigma\leq T}\|\bv^N_\sigma\|_{H_N^{-\alpha}}^p\right)\int_0^t\sup_{0\leq \sigma\leq s}\|\bu^N_\sigma-w^N_\sigma \|_{H_N^\beta}^p\, ds.\]
We use Gr\"onwall's inequality to obtain the upper bound
\begin{equation}\label{eq:cvg_w_beta}
 \sup_{0\leq \sigma\leq t}\|\bu^N_\sigma-w^N_\sigma \|_{H_N^\beta}^p \leq C\sup_{0\leq \sigma\leq T}\|Y^N_\sigma\|_{H_N^\beta}^p \exp\left(C\sup_{0\leq \sigma\leq T}\|\bv^N_\sigma\|_{H_N^{-\alpha}}^p\right).
\end{equation}
To conclude for the $H_N^\beta$ convergence, let us denote $\kappa_N = l(N)N^{-2\beta}/\log(N)$. For all $\varepsilon>0$, we have
\begin{align*}
\P\left(\sup_{0\leq \sigma\leq t}\|\bu^N_\sigma-w^N_\sigma \|_{H_N^\beta}^p>\varepsilon\right)
&\leq \P\left(C\sup_{0\leq \sigma\leq T}\|Y^N_\sigma\|_{H_N^\beta}^p \exp\left(C\sup_{0\leq \sigma\leq T}\|\bv^N_\sigma\|_{H_N^{-\alpha}}^p\right)>\varepsilon\right)\\
&\leq \P\left(C\sup_{0\leq \sigma\leq T}\|Y^N_\sigma\|_{H_N^\beta}>\varepsilon^{1/p}/\kappa_N^{1/p}\right)\\
& \quad+ \P\left( \sup_{0\leq \sigma\leq T}\|\bv^N_\sigma\|_{H_N^{-\alpha}}^p>C\log(\kappa_N)\right)\\
&\leq C\exp\left(C\log(N)(1-\varepsilon^{2/p}\kappa_N^{1-1/p})\right) + \frac{C}{\kappa_N}.\\
\end{align*} 
This proves the convergence for the $H_N^\beta$ topology. For the $L^\infty$ convergence, we have
\begin{align*}
\|\bu^N_t-w^N_t \|_{\infty}
&\leq \|Y^N_t\|_{\infty} +C\int_0^t \|T_N(t-s)\left(a(\bu^N_s-w^N_s)\bv^N_s+b-C(\bu^N_s)-b_C(w^N_s)\right)\|_{\infty}\, ds  \\
&\leq \|Y^N_t\|_{\infty} +\int_0^t \|T_N(t-s)\left(a(\bu^N_s-w^N_s)\bv^N_s\right)\|_{H_N^{1/2+\varepsilon}}\, ds\\
&\quad +\int_0^t\|b_C(\bu^N_s)-b_C(w^N_s)\|_\infty\, ds \\
&\leq \|Y^N_t\|_{\infty} +C\int_0^t (t-s)^{-1/2-\varepsilon}\|(\bu^N_s-w^N_s)\bv^N_s\|_{H_N^{-1/2-\varepsilon}}\, ds\\
&\quad +L\int_0^t\|\bu^N_s-w^N_s\|_\infty\, ds \\
&\leq \|Y^N_t\|_{\infty} +C\int_0^t (t-s)^{-1/2-\varepsilon}\|\bu^N_s-w^N_s \|_{H_N^\beta}\|\bv^N_s \|_{H_N^{-\alpha}}\, ds\\
&\quad +L\int_0^t\|\bu^N_s-w^N_s\|_\infty\, ds \\
&\leq \|Y^N_t\|_{\infty} +C\sup_{0\leq \sigma\leq T}\|\bu^N_\sigma-w^N_\sigma \|_{H_N^\beta}\sup_{0\leq \sigma\leq T}\|\bv^N_\sigma \|_{H_N^{-\alpha}} +L\int_0^t\|\bu^N_s-w^N_s\|_\infty\, ds \\
\end{align*}
Hence,  using the bound \eqref{eq:cvg_w_beta}, we have
\[\sup_{0\leq \sigma\leq t}\|\bu^N_\sigma-w^N_\sigma \|_{\infty}
\leq C\sup_{0\leq \sigma\leq T}\|Y^N_\sigma\|_{\infty}\exp\left(2C\sup_{0\leq \sigma\leq T}\|\bv^N_\sigma \|^p_{H_N^{-\alpha}}\right)+ L\int_0^t\sup_{0\leq \sigma\leq s}\|\bu^N_\sigma-w^N_\sigma \|_{\infty}\, ds.\] 
The prove ends with the same arguments as for the $H_N^\beta$ case.
\end{proof}

\subsection{Convergence of the discrete PDE}

The goal of this section is to prove that the solution $w^N$ of the discrete PDE \eqref{eq:wN} converges to the unique solution of the PDE
\begin{equation}\label{eq_w}
\frac{d}{dt}w=\Delta w+ R_C(w,v), \, w(0)=u_0.
\end{equation} 
The main difficulties come from the low regularity of the limit $v$. That is why we use some regularisation procedures to conclude.

\begin{theorem}\label{prop:wNw}
Let $0<\alpha<\beta<1/2$, $\tilde \beta>\beta$, $T\ge 0$, $(v^N)_{N\in\mathbb N}$ a sequence in $C([0,T];H^N)$ and 
	$v\in C([0,T];H^{-\alpha})$ such that 
\[v^N\to v, \text{when} \, N\to \infty \text{ in }  C([0,T];H^{-\alpha}).\]
Let $(u_0^N)_{N\in\mathbb N}$ a sequence in $H_N$ and $u_0\in H^{\tilde\beta}\cap C([0,1])$ such that:
\[ u_0^N\to u_0, \text{when} \, N\to \infty \text{ in }  H^{\tilde\beta}\cap L^\infty.\]
Then, there exists a unique solution $w\in C(\mathbb R^+,H^\beta\cap L^\infty)$ to \eqref{eq_w}.
For $N\in\mathbb N$, let $w^N$ be the solution of \eqref{eq:wN} . Then 
\[ w^N\to w, \text{when} \, N\to \infty \text{ in }  C([0,T]\times[0,1])\] 
and 
\[ \|w^N-P_Nw\|_{H^\beta} \to 0, \text{when} \, N\to \infty .\]	
\end{theorem}

\begin{remark}\label{r} Let us notice that in dimension greater than $1$, a similar results holds but we need stronger assumptions on 	the initial data $u_0$. This is due to product rules which are worse in higher dimension. Alternatively, we could prove 
 a convergence which does not hold up to time $0$ under similar assumptions on the initial data $0$. This is because we need to use the smoothing property 
of the heat semigroup on the initial data and this introduces a singularity in time. 

Similarly, up to the price of this convergence on positive time, we could get rid of the assumption that the initial 
data is continuous. 
\end{remark}
\hfill $\square$

\medskip
Since we use smoothing kernel in the proofs below. It is more convenient to work on the 
spatial interval $[-1/2,1/2]$. Since we work only at the level of partial differential equations
and difference equations, it is easy to switch from one interval to the other by the 
transformation $x\mapsto x-1/2$.

For instance, given a function $\varphi\in L^2(0,1)$, it is mapped to the
function $\psi(x)=\varphi(x-1/2)$. The new discrete space is the space of piecewise constant functions
which are constant on each $[(2l-1)/2N,(2l+1)/2N[$ for $l= -(N-1)/2, \dots , (N-1)/2$. The new projection 
from $L^2(-1/2,1/2)$ to the discrete space is defined by 
$$
\tilde P_N \psi_l =P_N \varphi_{(2l+N-1)/2}=N\int_{(2l-1)/2N}^{(2l+1)/2N}\psi(x)dx.
$$
Below we do not use tildes and use the same notations $H$, $H^\alpha$, $H^\alpha_N$ ... for the spaces on $[-1/2,1/2]$. Similarly we keep the
notations $\Delta$, $\Delta_N$, $T_N$ ... for the various operators. 

\bigskip

Let $(v^N)_{N\in \mathbb N}$ be a sequence in $C([0,T];H^N)$ and $v\in C([0,T];H^{-\alpha})$ for some $\alpha >0$ be such that $v^N\to v$ in $C([0,T];H^{-\alpha})$. 

We consider the equation:
\begin{equation*}
	\frac{d}{dt}w=\Delta w+ R_C(w,v), \, w(0)=u_0
\end{equation*}
with periodic boundary conditions on $[-1/2,1/2]$. Recall that $R_C(u,v)=a_Cuv+b_C(u)+d_C v$ and it satisfies Assumption \ref{ass:RC}.

\begin{prop}\label{prop_w}
	Let $\beta \in (\alpha, 1/2)$ and assume that 
	$u_0\in H^\beta\cap L^\infty$, then there exists a unique solution to \eqref{eq_w} in 
	$C(\mathbb R^+,H^\beta\cap L^\infty)$.
\end{prop}
\begin{proof}
	The proof follows by a truncation argument and a fixed point argument on the mild form of the equation. 

	Take a smooth function $\theta$ with compact support in $[-2,2]$ and equal to $1$ on $[-1,1]$, define 
	$b_C^M(x) = \theta(x/M_1)$, with $M_1=2\max\{\| u_0^N\|_\infty, M\}$ and $R_C^M$ as $R_C$ replacing $b_C$ by $b_C^M$.

	Define $\mathcal T\,:\, C([0,T],H^\beta)\to C([0,T],H^\beta)$ by
	$$
	\mathcal Tw(t) = e^{\Delta t} u(0) + \int_0^t e^{\Delta(t-s)}R_C^M(w(s), v(s))\, ds.
	$$
	By product rules in Sobolev spaces , we know that for $\gamma> \alpha+1/2-\beta$:
	$$
	\|wv\|_{H^{-\gamma}} \le C \|w\|_{H^\beta}\|v\|_{H^{-\alpha}}.
	$$
	We choose $3/2>\gamma> \alpha+1/2-\beta$.
	
	Note that since $b_C$ is Lipschitz on bounded sets, $b_C^M$ is Lipschitz on $\mathbb R$.
	
	Using the smoothing property of the heat semigroup, we deduce, for $w_1,\, w_2 \in C([0,T],H^\beta\cap L^\infty)$:
\begin{align*}
\|\mathcal T w_1(t)-\mathcal T w_2(t)\|_{H^\beta} 
\le& \| u(0)\|_{H^\beta} + C \int_0^t |t-s|^{-\beta/2}\|w_1(s)-w_2(s)\|_{L^\infty}\, ds\\
& + C \int_0^t |t-s|^{-(\gamma+\beta)/2}	\|w_1(s)-w_2(s)\|_{H^\beta}\|v(s)\|_{H^{-\alpha}}\, ds\\
\end{align*}
	and
	
\begin{align*}
\|\mathcal T w_1(t)-\mathcal T w_2(t)\|_{L^\infty} 
\le& \| u(0)\|_{L^\infty} + C \int_0^t \|w_1(s)-w_2(s)\|_{L^\infty}\, ds\\
& + C \int_0^t |t-s|^{-(\gamma+1/2)/2} \|w_1(s)-w_2(s)\|_{H^\beta}\|v(s)\|_{H^{-\alpha}}\, ds\\
\end{align*}

Since $v\in C(\mathbb R^+; H^{-\alpha})$, we deduce that for $T$ sufficiently small, depending only on the Lipschitz constant of $b_C^M$ and $v$, $\mathcal T$ is a contraction. It has therefore a unique fixed point $w$ which is a solution 
	of the truncated equation where $b_C^M$ replaces $b_C$. Iterating this construction, we get rid of the smallness condition on $T$.

	By the maximum principle and Assumption \ref{ass:RC}, we know that $w$ is bounded in the $L^\infty$ norm by $\max\{\| u_0^N\|_\infty, M\}$. Thus 
	$$
	b_C^M(w(t))=	b_C(w(t), t\ge 0,
	$$
	and $w$ is a solution of the original equation without the truncation.

\end{proof}

Let $(w_N)_{N\in \mathbb N}$ be the solutions of Equations \eqref{eq:wN} in $H^N$:
\begin{equation}\label{eq_wN}
\frac{d}{dt}w_N=\Delta_N w_N + R_C(w_N,v_N), \, w_N(0)=u_0^N.
\end{equation}
By Cauchy-Lipschitz theorem there exists a local in time solution and, by the maximum principle and Assumption \ref{ass:RC}, it is global.

We introduce a smoothing kernel $(\rho^\epsilon)_{\epsilon}$ defined classically, for $\epsilon>0$, by: 
$$
\rho^{\epsilon}(y)=\epsilon^{-1}\rho(\frac{y}{\epsilon}),\ y\in [-1/2,1/2],
$$
for a smooth compactly supported even non negative $\rho$ such that $\int_{\mathbb{R}}\rho(z)dz=1$ and whose support is 
included in $](1/2,1/2)$. All these functions can be seen as smooth periodic functions on $[-1/2,1/2]$. We also write $\rho^{\epsilon,N}=P_N \rho^{\epsilon}$.

We then define $v_\epsilon=\rho^{\epsilon}*v,\; v^{\epsilon,N}=\rho^{\epsilon,N}*_{N}v^N\; \epsilon>0,\; N\in \mathbb{N}$, where, for two 
functions in $a,b\in H_N$, we write $(a*_N b )_k=\frac1N \sum_{\ell=-(N-1)/2}^{(N-1)/2} a_{k-\ell}b_k$, and consider $u^\epsilon$, $w^{\epsilon,N}$ satisfying:
$$
w^\epsilon(t) = e^{\Delta t}\rho^\epsilon * u(0) + \int_0^t e^{\Delta(t-s)} R_C(w^\epsilon(s), v^\epsilon(s))\, ds,
$$
\[w^{\epsilon,N}(t) = T_N(t)(\rho^{\epsilon,N} *_N u^N_0) + \int_0^t T_N(t-s)R_C(w^{\epsilon,N}(s), v^{\epsilon,N}(s))\, ds.\]

The proof of the existence of $w^\epsilon$, $w^{\epsilon,N}$ is the same as for $w$ and $w^N$. They satisfy uniform bounds.

Let us first remark that by the maximum principle, we already know that all these functions are uniformly bounded in the supremum norm by $\max\{\| u_0^N\|_\infty, M\}$.

 We need the following technical lemma whose counterpart in the continuous case is classical.

\begin{lemme}\label{rho_N}
	Let $\varphi\in H_N$ then for any $\epsilon>0$, $\alpha\in [-1,1]$:
	$$
	\|\rho^{\epsilon,N} *_N \varphi\|_{H^\alpha_N} \le  \| \varphi\|_{H^\alpha_N}
	$$
	and for $\alpha,\, \beta \in [-1,1]$ such that $0\le \beta-\alpha\le 1$:
	$$
	\|\rho^{\epsilon,N} *_N \varphi -\varphi\|_{H^\alpha_N} \le C_{\beta,\alpha}\epsilon^{\beta-\alpha} \| \varphi\|_{H^\beta_N}
	$$
\end{lemme}
\begin{proof}
	We first consider $\alpha = 0,\, 1$. For $\alpha = 0$, we have, using $	\frac{1}{N}\sum_{k=0}^{N-1}\rho^{\epsilon,N}_{k}=1$, 
	$$
	\|\rho^{\epsilon,N} *_N \varphi\|_{H_N}^2=\frac{1}{N^3}\sum_{k=0}^{N-1}|\sum_{l=0}^{N-1}\rho^{\epsilon,N}_{k-l}\varphi_l|^2\le
	\frac{1}{N^2}\sum_{k=0}^{N-1}\sum_{l=0}^{N-1}\rho^{\epsilon,N}_{k-l}|\varphi_l|^2 \le 
	\frac1{N}\sum_{l=0}^{N-1}|\varphi_l|^2 = \| \varphi\|_{H_N}^2.
	$$
	We then write:
	$$
	\|\nabla_N^+(\rho^{\epsilon,N} *_N \varphi)\|_{H_N}^2
	=\|\rho^{\epsilon,N} *_N \nabla_N^+\varphi\|_{H_N}^2 \le \|\nabla_N^+\varphi\|_{H_N}^2	$$
	using the result for $\alpha=0$. The result follows for $\alpha = 0,\, 1$. By interpolation, it also follows for $\alpha\in [0,1]$.
	
	We then use duality for $\alpha\in [-1,0]$. Take $\varphi\in H_N$ and write:
	$$
	\|\rho^{\epsilon,N} *_N \varphi\|_{H^\alpha_N}=\sup_{\|\psi\|_{H^{-\alpha}_N}=1} (\rho^{\epsilon,N} *_N \varphi,\psi)_{H_N}=\sup_{\|\psi\|_{H^{-\alpha}_N}=1} ( \varphi,\rho^{\epsilon,N} *_N\psi)_{H_N}
	\le \| \varphi\|_{H^\alpha_N}.
	$$
	This proves the first point. 
	
	For the second, we first consider $\alpha =0$ and $\beta=1$ and write:
	\begin{align*}
		\|\rho^{\epsilon,N} *_N \varphi -\varphi\|_{H_N}^2&=\frac1{N^3}\sum_{k=-(N-1)/2}^{(N-1)/2}\big|\sum_{l=-(N-1)/2}^{(N-1)/2}  \rho^{\epsilon,N}_l (\varphi_k-\varphi_{k-l})   \big|^2\\
		&\le \frac1{N^2}\sum_{k=-(N-1)/2}^{(N-1)/2}\sum_{l=-(N-1)/2}^{(N-1)/2}  \rho^{\epsilon,N}_l \big| (\varphi_k-\varphi_{k-l})   \big|^2\\
		&\le  \frac1{N^4}\sum_{k=-(N-1)/2}^{(N-1)/2}\sum_{l=-(N-1)/2}^{(N-1)/2}  \rho^{\epsilon,N}_l \big| \sum_{m=k-l}^{k-1}(\nabla_N^+\varphi)_m  \big|^2\\
		&\le  \frac1{N^4}\sum_{k=-(N-1)/2}^{(N-1)/2}\sum_{l=-(N-1)/2}^{(N-1)/2}  l\rho^{\epsilon,N}_l  \sum_{m=k-l}^{k-1}\big|(\nabla_N^+\varphi)_m  \big|^2\\
		&\le  \frac1{N^4}\sum_{m=-(N-1)/2}^{(N-1)/2}\sum_{l=-(N-1)/2}^{(N-1)/2}  l^2\rho^{\epsilon,N}_l  \big|(\nabla_N^+\varphi)_m  \big|^2\\
		&\le  C\epsilon^2 \frac1{N}\sum_{m=-(N-1)/2}^{(N-1)/2} \big|(\nabla_N^+\varphi)_m  \big|^2\\
		&=C\epsilon^2 |\varphi|^2_{H^1_N}.
	\end{align*}
	We have used:
	$$
	\frac1{N^3}\sum_{l=-(N-1)/2}^{(N-1)/2}  l^2\rho^{\epsilon,N}_l=\frac1N \sum_{l=-(N-1)/2}^{(N-1)/2}  \left(\frac{l}{N}\right)^2\rho^{\epsilon,N}_l
	\le C \int_{-1/2}^{1/2}x^2 \rho^\epsilon(x) dx= C\epsilon^2 \int_{-1/2}^{1/2}x^2 \rho(x) dx.
	$$
	The result is clearly true for $\alpha=\beta$ and we deduce the general case for $\alpha,\, \beta$ non negative by interpolation the $(0,1)$ and 
	$(\alpha/(1-\alpha+\beta),\alpha/(1-\alpha+\beta))$. 
	
	
	For the general case, we write, using the commutation of the discrete convolution and $\Delta_N$:
	\begin{align*}
		&\|\rho^{\epsilon,N} *_N \varphi -\varphi\|_{H_N^\alpha}= \|(I-\Delta_N)^{\alpha/2}(\rho^{\epsilon,N} *_N \varphi -\varphi)\|_{H_N}\\
		&\le C \epsilon^{\beta-\alpha} \|(I-\Delta_N)^{\alpha/2}\varphi\|_{H_N^{\beta-\alpha}} =C \epsilon^{\beta-\alpha} \|\varphi\|_{H_N^\beta}.
	\end{align*} 
	
\end{proof}
\begin{lemme}\label{l5.3}
	Let $\beta \in (\alpha, \frac12)$ and assume that 
	$u_0\in H^\beta\cap L^\infty$, then for any $T>0$, $(w^\epsilon)_{\epsilon>0}$ is uniformly bounded in $C([0,T],H^\beta)$.
	
	Moreover, then for any $T>0$, $(w^N)_{N\in \mathbb N}$ and $(w^{\epsilon,N})_{\epsilon >0,\, N\in \mathbb N}$ are uniformly bounded with respect to $\epsilon$ and $N$ in $C([0,T],H^\beta_N)$.
\end{lemme}
\begin{proof}
	Using the same arguments as in the proof of Proposition \ref{prop_w}, we have, with $3/2>\gamma> \alpha+1/2-\beta$:
	$$
	\| w^\epsilon(t)\|_{H^\beta} \le  \| u(0)\|_{H^\beta} + C \int_0^t |t-s|^{-\beta/2}M_1 + |t-s|^{-(\gamma +\beta)/2}
	(\|w^\epsilon(s)\|_{H^\beta} +1)\|v^\epsilon(s)\|_{H^{-\alpha}}\, ds,
	$$
	and Gr\"onwall's inequality indeed imply a uniform bounds on $w^\epsilon$ in $C([0,T],H^\gamma)$ since $v^\epsilon$ is uniformly bounded in 
	$C([0,T],H^{-\alpha})$.
	
	The proof of the other points follows from the same computations since $T_N$ satisfies the same smoothing property as the heat kernel in $H_N$, 
	$v^{\epsilon,N}$ is uniformly bounded in 
	$C([0,T],H^{-\alpha}_N)$ thanks to Lemma \ref{rho_N}, by the convergence of $v_N$ to $v$ in
	$C([0,T],H^{-\alpha})$, and the equivalence of the discrete and continuous Sobolev norm on $H^{-\alpha}_N$ (\cite{Blountphd}).
\end{proof}

\begin{lemme}
	Let $\beta \in (\alpha,\frac12)$ and $0<\lambda<1+\beta-\alpha$ 
	$$
	\sup_{[0,T]}	\|w^\epsilon(t)-w(t)\|_{H^\beta}\le \|u_0-\rho_\epsilon *u_0\|_{H^{\beta}}+C \epsilon^{\lambda}\sup_{[0,T]}\|v(t)\|_{H^{-\alpha}},
	$$
	$$
	\sup_{[0,T]}	\|w^\epsilon(t)-w(t)\|_{L^\infty}\le \|u_0-\rho_\epsilon *u_0\|_{L^\infty}+C \epsilon^{\lambda}\sup_{[0,T]}\|v(t)\|_{H^{-\alpha}}
	$$
	and for $N\in\mathbb N$:
	$$
	\sup_{[0,T]}	\|w^{\epsilon,N}(t)-w^N(t)\|_{H_N^\beta}\le \|u_0^N-\rho_\epsilon^N *_N u_0^N\|_{H_N^{\beta}}+C \epsilon^{\lambda}\sup_{[0,T]}\|v^N(t)\|_{H_N^{-\alpha}},
	$$
	$$
	\sup_{[0,T]}	\|w^{\epsilon,N}(t)-w^N(t)\|_{L^\infty}\le \|u_0^N-\rho_\epsilon^N *_N u_0^N\|_{L^\infty}+C \epsilon^{\lambda}\sup_{[0,T]}\|v^N(t)\|_{H_N^{-\alpha}}
	$$
	for a constant $C$ independant on $\epsilon$ or $N$.
\end{lemme}

\begin{proof}
	Define $r^\epsilon=w^\epsilon-w$, then we may write
	\begin{align*}
		&r^\epsilon(t)=e^{\Delta t}(\rho^\epsilon*u_0 -u_0)\\
		&+\int_0^t e^{\Delta (t-s)}\left((b_C(w(s))-b_C(w^\epsilon(s))) +a_C (w(s)v(s)-w^\epsilon(s)v^\epsilon(s))+d_C(v(s)-v^\epsilon(s))\right)ds 
	\end{align*}
	Since $w$ and $w^\epsilon$ are uniformaly bounded, we may assume that $b_C$ is a Lipschitz function, we deduce:
	$$
	\|b_C(w(s))-b_C(w^\epsilon(s)\|_{L^2}\le L\|w(s)-w^\epsilon(s)\|_{L^2}\le  L\|w(s)-w^\epsilon(s)\|_{H^\gamma}.
	$$
	We  take $3/2>\gamma> \alpha+\lambda+1/2-\beta$. By the product rules in Sobolev spaces (indiquer le résultat):
	
\begin{align*}
\|w(s)v(s)-w^\epsilon(s)v^\epsilon(s)\|_{H^{-\gamma}}
\le & C \|w(s)-w^\epsilon(s)\|_{H^\beta}\|v(s)\|_{H^{-\alpha}}\\
& +C\|w^\epsilon(s)\|_{H^\beta}\|v(s)-v^\epsilon(s)\|_{H^{-\alpha-\lambda}}\\
\end{align*}

	so that by Lemma \ref{l5.3}:
\[\|w(s)v(s)-w^\epsilon(s)v^\epsilon(s)\|_{H^{-\gamma}}\le C\left( \|w(s)-w^\epsilon(s)\|_{H^\beta}+\|v(s)-v^\epsilon(s)\|_{H^{-\alpha-\lambda}}\right).\]

We then use the smoothing property of the heat semigroup to write:
\begin{align*}
\|r^\epsilon(t)\|_{H^\beta}
\le &  \|u_0-\rho_\epsilon *u_0\|_{H^{\beta}} +C\int_0^t |t-s|^{-(\gamma+\beta)/2}\left(\|r^\epsilon(s)\|_{H^\beta}+\|v(s)-v^\epsilon(s)\|_{H^{-(\alpha+\lambda}}\right)ds\\
\le & \|u_0-\rho_\epsilon *u_0\|_{H^{\beta}} + C \epsilon^{\lambda}T^{1-(\gamma+\beta)/2}\sup_{[0,T]}\|v(s)\|_{H^{-\alpha}})\\ 
& +C\int_0^t |t-s|^{-(\gamma+\beta)/2}\|r^\epsilon(s)\|_{H^\beta}ds.\\
\end{align*}
	The result by Gr\"onwall's inequality. 
	
	The other results are proved similarly, in particular thanks to the previous Lemmas which give all the property we need to obtain constants uniform in $N$. 
\end{proof}

\begin{lemme}\label{l5.5}
	$$
	\|v^{\epsilon,N}-P_Nv^{\epsilon}\|_{L^2}\le C_\epsilon(\frac{1}{N}+\|v^N-v\|_{H^{-\alpha}}). 
	$$
\end{lemme}
\begin{proof}

	Let us split $v^{\epsilon,N}-P_N v^{\epsilon}$ into $\rho^{\epsilon,N}*_{N}v^N-\rho^\epsilon*v^N$, $(I-P_N )v^\epsilon$ and 
	$\rho^\epsilon*(v^N-v)$. Then, thanks to the smoothing property of the convolution:
	$$
	\|\rho^\epsilon*(v^N-v)\|_{L^2}\le C(\epsilon,\alpha)\|v_N-v\|_{H^{-\alpha}},
	$$
	and by Equation \eqref{truc2}
	$$
	\|(I-P_N )v^\epsilon\|_{L^2}\le \frac {C}N  \|v^\epsilon\|_{H^1} \le \frac{C_\epsilon}N. 
	$$
	Let $k=-(N-1)/2,\dots , (N-1)/2$ and $x\in [(2k-1)/2N, (2k+1)/2N)$, and write
	\begin{align*}
		|\rho^{\epsilon,N}*_{N}v^N(x)-\rho^\epsilon*v(x)|&=\left|\sum_{l=-(N-1)/2}^{(N+1)/2}a_l(k,x) v^N_l \right|\\ \\
		&\le \|a(k,x)\|_{H_N^1}\|v^N\|_{H^{-1}_N},
	\end{align*}
	with 
	$$
	a(k,x)= N \int_{(2l-1)/2N}^{(2l+1)/2N}
	\rho^{\epsilon}(k/N-y)-\rho^{\epsilon}(x-y)dy.
	$$
	Then:
	\begin{align*}
		\|a(k,x)\|_{H_N}^2&=\frac1N \sum_{l=-(N-1)/2}^{(N+1)/2} |N \int_{(2l-1)/2N}^{(2l+1)/2N}
		\rho^{\epsilon}(k/N-y)-\rho^{\epsilon}(x-y)dy|^2\\
		& \le \frac1{4N^2}L_{\epsilon}^2,
	\end{align*}
	where $L_\epsilon$ is the Lipschitz constant of $\rho_\epsilon$. Moreover
	\begin{align*}
		\|\nabla_N^+ a(k,x)\|_{H_N}^2&=\frac1N \sum_{l=-(N-1)/2}^{(N+1)/2} |N^2 \int_{(2l-1)/2N}^{(2l+1)/2N}&
		\rho^{\epsilon}((k-1)/N-y)-\rho^{\epsilon}(x-1/N-y)\\
		& &-\rho^{\epsilon}(k/N-y)+\rho^{\epsilon}(x-y)dy|^2\\
		&\le  \frac1{4N^2}L_{1,\epsilon}^2,
	\end{align*}
	where $L_{1,\epsilon}$ is the Lipschitz constant of $\rho_\epsilon'$. We deduce:
	$$
	|\rho^{\epsilon,N}*_{N}v^N(x)-\rho^\epsilon*v(x)|\le \frac{C_\epsilon}{N}  \|v^N\|_{H^{-1}_N}.
	$$
	Since $v^N$ is bounded in $H^{-1}_N$ and since the supremum norm is larger than the $L^2$ norm, we deduce:
	$$
	\|v^{\epsilon,N}-P_Nv^{\epsilon}\|_{L^2}\le C_\epsilon(\frac{1}{N}+\|v_N-v\|_{H^{-\alpha}}). 
	$$
	
\end{proof}
\begin{lemme}
For all $N\in \mathbb N$, $t\in [0,T]$:
$$
\|w^{\epsilon,N}(t)-P_Nw^{\epsilon}(t)\|_{H^{\beta}_N}	 \le C_\epsilon(\frac1N+\|u_0^N-u_0\|_{L^\infty}+\sup_{[0,T]}. \|v^N(s)-v(s)\|_{H^{-\alpha}}).
$$
and 
$$
\|w^{\epsilon,N}(t)-w^{\epsilon}(t)\|_{L^\infty}	 \le C_\epsilon(\frac1N+\|u_0^N-u_0\|_{L^\infty}+\sup_{[0,T]}. \|v^N(s)-v(s)\|_{H^{-\alpha}}).
$$
\end{lemme}
\begin{proof}
	Let us introduce $\widetilde w^{\epsilon,N}=P_N w^{\epsilon}$ and $r^{\epsilon,N}=w^{\epsilon,N}-
	\widetilde w^{\epsilon,N}$. Then:
	\begin{align*}
		r^{\epsilon,N}(t) =& T_N(t)(\rho^{\epsilon,N} *_N u^N_0-P_N \rho^\epsilon * u_0) + \int_0^t T_N(t-s)\bigg[b_C(w^{\epsilon,N}(s)) - P_Nb_C(w^\epsilon(s)) \\
		& + a_C(w^{\epsilon,N}(s) v^{\epsilon,N}(s)-P_N(w^\epsilon (s) v^\epsilon(s))) + d_C(v^{\epsilon,N}(s)-P_N v^\epsilon(s))\bigg]\, ds\\
		&+ \int_0^t T_N(t-s)\bigg[(\Delta_N P_N - P_N\Delta)w^\epsilon \bigg]\, ds.
	\end{align*}
	Again, since $w^{\epsilon,N}$ and $w^\epsilon$ are uniformly bounded in the supremum norm, we may assume that $b_C$ is Lipschitz and write, thanks to Equation \eqref{truc2},
	
\begin{align*}
\|b_C(w^{\epsilon,N}(s)) - P_Nb_C(w^\epsilon(s))\|_{H_N}
\le & C \|r^{\epsilon,N}(s))\|_{H_N} + \|(I - P_N)b_C(w^\epsilon(s))\|_{H_N}\\
& + \|(I - P_N)w^\epsilon(s)\|_{H_N}\\
\le & C \|r^{\epsilon,N}(s))\|_{H_N} + \frac{C}{N} \|w^\epsilon(s)\|_{H^1} \\
\le & C \|r^{\epsilon,N}(s))\|_{H_N} + \frac{C_\epsilon}{N}.
\end{align*}

Let us estimate the second term in the integral above:
\begin{align*}
\|w^{\epsilon,N} v^{\epsilon,N}-P_N(w^\epsilon v^\epsilon)\|_{H_N}
\le & \|r^{\epsilon,N} v^{\epsilon,N}\|_{H_N}+\|\widetilde w^{\epsilon,N} (v^{\epsilon,N}-P_Nv^\epsilon)\|_{H_N}\\
&+\|\widetilde w^{\epsilon,N} P_Nv^\epsilon- P_N(w^\epsilon v^\epsilon)\|_{H_N}\\
\le & C_\epsilon \|r^{\epsilon,N} \|_{H_N}+C_\epsilon \left(\frac{1}{N}+\|v^N-v\|_{H^{-\alpha}}\right)\\
&+\|\widetilde w^{\epsilon,N} P_Nv^\epsilon- P_N(w^\epsilon v^\epsilon)\|_{H_N}.
\end{align*}

We have used Lemma \ref{l5.5} and that the discrete and continuous convolution have a smoothing effect uniform in $N$ so that $v^{\epsilon,N}$ and $\widetilde w^{\epsilon,N}$ are uniformly bounded with respect to $N$ in $H^1$ and $H^1_N$. Moreover, since $w^\epsilon$ and $v^\epsilon$ are smooth, it can be checked that 
	$$
	\|\widetilde w^{\epsilon,N} P_Nv^\epsilon- P_N(w^\epsilon v^\epsilon)\|_{H_N}\le \frac{C}{N}\|(w^\epsilon)'\|_{L^\infty} \|v^\epsilon\|_{L^\infty}\le \frac{C_\epsilon}N.
	$$
	Moreover, on $[(2l-1)/N,(2l+1)/N)$:
$$
(P_N\Delta -\Delta-NP_N)w^\epsilon=N\int_{(2l-1)/2N}^{(2l+1)/N} (w^\epsilon)''(x)-N^2(w^\epsilon(x-1/N)+w^\epsilon(x+1/N)-2w(x))dx.
$$ 
One can check by Taylor formula that:
$$
|(w^\epsilon)''(x)-N^2(w^\epsilon(x-1/N)+w^\epsilon(x+1/N)-2w(x))|\le C \|(w^{(4)})^\epsilon\|_{L^\infty}/N^2\le C_\epsilon/N^2.
$$
Hence
$$
\|(P_N\Delta -\Delta-NP_N)w^\epsilon\|_{H_N}\le C_\epsilon/N^2.
$$
	It follows, thanks to the smoothing property of the discrete heat kernel,
\begin{align*}
\|r^{\epsilon,N}(t)\|_{H^{\beta}_N}	 
\le & \|\rho^{\epsilon,N} *_N u^N_0-P_N \rho^\epsilon * u_0\|_{H^{\beta}_N}\\
& + C_\epsilon \int_0^t |t-s|^{-\beta/2}\bigg[\|r^{\epsilon,N}(s))\|_{L^2} + \frac1{N}+\|v^N(s)-v(s)\|_{H^{-\alpha}}\bigg]\, ds.
\end{align*}

By similar arguments as in Lemma \ref{l5.5}: 
\[\|\rho^{\epsilon,N} *_N u^N_0-P_N \rho^\epsilon * u_0\|_{H^{\beta}_N}\le C_\epsilon\left(\frac1N+\|u_0^N-u_0\|_{L^\infty}\right),\]

so that, since $\|r^{\epsilon,N}(t)\|_{H_N}\le \|r^{\epsilon,N}(t)\|_{H^{\beta}_N}$, we have by Gr\"onwall's inequality:
\[\|r^{\epsilon,N}(t)\|_{H^{\beta}_N}	\le C_\epsilon\left(\frac1N+\|u_0^N-u_0\|_{L^\infty}+\sup_{[0,T]}. \|v^N(s)-v(s)\|_{H^{-\alpha}}\right).\]

By similar arguments, we have:
\[\|r^{\epsilon,N}(t)\|_{L^\infty} \le C_\epsilon\left(\frac1N+\|u_0^N-u_0\|_{L^\infty}+\sup_{[0,T]} \|v^N(s)-v(s)\|_{H^{-\alpha}}\right).\]

Finally, since $\|\widetilde w^{\epsilon,N}-w^{\epsilon}\|_{L^\infty}\le \frac{C_\epsilon}N$, we deduce the result.
\end{proof}

Gathering the above results proves the Theorem choosing first $\epsilon$ small and the $N$ large since under the assumptions 
on the initial data $\|u_0^N-\rho_\epsilon^N *_N u_0^N\|_{L^\infty}\to 0$ and $\|u_0^N-\rho_\epsilon^N *_N u_0^N\|_{H^\beta_N}\to 0$ uniformly in $N$ when $\epsilon \to 0$.

\subsection{Uniqueness of the limit point and conclusion}

According to Proposition \ref{prop:uNwN} ad Theorem \ref{prop:wNw}, it follows that $\bu^N$ converges in probability to $w$, on $\DD\left([0,T], L^\infty\cap H_N^\beta\right)$,  for the uniform topology. Therefore, we can take the limit in the equation defining $\bv^N$. It results that the limit point $v$ satisfies the ODE-part of System \eqref{eq:uv}.

\begin{lemme}
Assume that $(v^N_0)_{N\geq 1}$ converges to $v_0$ in $H^{-\alpha}$. Then, the limit point $v$ satisfies 
\[\partial_t v(t) = R_D(w(t), v(t)),\, v(0) = v_0,\]
with $w$ the solution of Equation \eqref{eq_w}.
\end{lemme}

\begin{proof}
For all $N\geq1$ and $0\leq t\leq T$, we have
\[\bv^N(t) = v^N(0) + \int_0^t d_D\bv^N(s) + b_D(\bu^N(s))\, ds+Z^N_D(t\wedge \tau).\]

As $v$ is a limit point of $\bv^N$ and by assumption on the initial data, is is sufficient to show that $b_D(\bu)$ converge for $H^{-\alpha}$. Actually, we have 
\begin{align*}
\sup_{0\leq t\leq T}\|b_D(\bu^N(t)) - b_D(w^N(t))\|_{H^{-\alpha}}	
\leq & \sup_{0\leq t\leq T}\|b_D(\bu^N(t)) - b_D(w^N(t))\|_{\infty}\\	
\leq & C\sup_{0\leq t\leq T}\|\bu^N(t) - w^N(t)\|_{\infty},
\end{align*}	
as $b_D$ is locally Lipschitz continuous, and $\bu^n$ and $w^N$ are bounded. This ends the proof.
\end{proof}

To ensure that $\bv^N$ admits only one limit point, we need to show that the equation \eqref{eq:uv} has only one solution in $\CC\left([0,T], (L^\infty\cap H^\beta)\times H^{-\alpha}\right)$.
\begin{lemme}
For all $(u_0,v_0)\in (L^\infty\cap H^\beta) \times H^{-\alpha}$, there exists only one solution to the system \eqref{eq:uv}, with periodic boundary conditions and initial condition $u(0) = u_0$ and $v(0) = v_0$, in $\CC\left([0,T], (L^\infty\cap H^\beta)\times H^{-\alpha}\right)$.
\end{lemme}
\begin{proof}
Let $(u_1,v_1)$ and $(u_2, v_2)$ two solutions of Equation \eqref{eq:uv} in $\CC\left([0,T], (L^\infty\cap H^\beta)\times H^{-\alpha}\right)$, with the same initial data $(u_0,v_0)$. Let $0<\varepsilon,\gamma$, such that $1/2+\varepsilon+\gamma<2$ and $-\gamma<1/2+\varepsilon-\alpha-1/2$. For all $0\leq t\leq T$, we have

\begin{align*}
\|u_1(t)-u_2(t)\|_{H^\beta} 
\leq & \int_0^t\| e^{\Delta (t-s)}a_c(u_1(s)v_1(s)-u_2(s)v_2(s)))\|_{H^\beta}\, ds\\
&+ \int_0^t\| e^{\Delta (t-s)}d_C(v_1(s)-v_2(s)))\|_{H^\beta}\, ds\\
& +\int_0^t\| e^{\Delta (t-s)}( b_C(u_1(s))-b_C(u_2(s)))\|_{H^\beta}\, ds\\
\leq & C\int_0^t (t-s)^{-(\beta+\gamma)/2}\|u_1(s)v_1(s)-u_2(s)v_2(s)\|_{H^{-\gamma}}\, ds\\
& + C\int_0^t (t-s)^{-(\beta+\alpha)/2}\|v_1(s)-v_2(s)\|_{H^{-\alpha}}\, ds\\
& +C\int_0^t (t-s)^{-\beta/2}\|b_C(u_1(s))-b_C(u_2(s))\|_{L^2}\, ds\\
\leq & C\int_0^t (t-s)^{-(\beta+\gamma)/2}\|u_1(s)\|_{H^{\beta}}\|v_1(s)-v_2(s)\|_{H^{-\alpha}}\, ds\\
& + C\int_0^t (t-s)^{-(\beta+\gamma)/2}\|(u_1(s)-u_2(s)\|_{H^{\beta}}\|v_2(s)\|_{H^{-\alpha}}\, ds\\
& + C\int_0^t (t-s)^{-(\beta+\alpha)/2}\|v_1(s)-v_2(s)\|_{H^{-\alpha}}\, ds\\
& +C\int_0^t (t-s)^{-\beta/2}\|u_1(s)-u_2(s)\|_{\infty}\, ds\\
\leq & + C\int_0^t (t-s)^{-(\beta+\gamma)/2}\|(u_1(s)-u_2(s)\|_{H^{\beta}}\, ds\\
& + C\int_0^t (t-s)^{-(\beta+\alpha)/2}\|v_1(s)-v_2(s)\|_{H^{-\alpha}}\, ds\\
& +C\int_0^t (t-s)^{-\beta/2}\|u_1(s)-u_2(s)\|_{\infty}\, ds\\
\end{align*} 
For the norm $\|\cdot\|_\infty$, we obtain 
\begin{align*}
\|u_1(t)-u_2(t)\|_{\infty} 
\leq & + C\int_0^t (t-s)^{-(1/2+\varepsilon+\gamma)/2}\|(u_1(s)-u_2(s)\|_{H^{\beta}}\, ds\\
& + C\int_0^t (t-s)^{-(1/2+\varepsilon+\alpha)/2}\|v_1(s)-v_2(s)\|_{H^{-\alpha}}\, ds\\
& +C\int_0^t (t-s)^{-(1/2+\varepsilon)/2}\|u_1(s)-u_2(s)\|_{\infty}\, ds\\
\end{align*}
Lastly, we control the difference between $v_1$ and $v_2$ with
\[\|v_1(t)-v_2(t)\|_{H^{-\alpha}}
\leq   C\int_0^t \|v_1(s)-v_2(s)\|_{H^{-\alpha}}\, ds
 +C\int_0^t \|u_1(s)-u_2(s)\|_{\infty}\, ds\]
Then, we have proved that
\begin{align*}
&\|(u_1(t),v_1(t))-(u_2(t), v_2(t))\|_{(L^\infty\cap H^\beta)\times H^{-\alpha}} \\
&\quad\le C\int_0^t(t-s)^{-(1/2+\varepsilon+\gamma)/2}\|(u_1(s),v_1(s))-(u_2(s), v_2(s))\|_{(L^\infty\cap H^\beta)\times H^{-\alpha}}\, ds.
\end{align*} 
The proof is concluded by a singular version of Gr\"onwall's inequality, as in Proposition \ref{prop:uNwN}.
\end{proof}

Now, we have the last arguments to conclude the proof of our main result.

\begin{proof}[Proof of Theorem \ref{prop:main}]
Let $(u,v)$ the unique solution of Equation \eqref{eq:uv}, associated to $(u_0,v_0)$. The sequence $(\bv^N)_N$ is tight and from the previous results, it admits a unique limit point, $v$. Then, $\bv^N$ converges in distribution to this unique limit point. Therefore, it implies that $\bu^N$ converges to $u$. And as the limit $(u,v)$ is deterministic, we prove that $(\bu^N, \bv^N)$ converges in probability to $(u,v)$. 

To conclude, let us prove the convergence without truncation. For all $0<\varepsilon<1$, on the event 
\[\left\{\sup_{0\leq t\leq T}\|(\bu^N(t), \bv^N(t)) - (u(t), v(t))\|_{(L^\infty\cap H_N^\beta)\times H^{-\alpha}} \leq \varepsilon\right\}\]
we have $\sup_{0\leq t\leq T}\|(\bu^N(t), \bv^N(t))\|_{(L^\infty\cap H_N^\beta)\times H^{-\alpha}}\leq M+1$. This implies that $\tau_N\geq T$, and that $(u^N, v^N)$ is well-defined on $[0,T]$ and is equal to the process $(\bu^N, \bv^N)$. Hence, we have
\begin{align*}
&\P\left( \sup_{0\leq t\leq T}\|(u^N(t), v^N(t)) - (u(t), v(t))\|_{(L^\infty\cap H_N^\beta)\times H^{-\alpha}}>\varepsilon\right) \\
& \leq \P\left( \sup_{0\leq t\leq T}\|(\bu^N(t), \bv^N(t)) - (u(t), v(t))\|_{(L^\infty\cap H_N^\beta)\times H^{-\alpha}}>\varepsilon\right)\\
& \xrightarrow[N\to\infty]{} 0.\\
\end{align*}
This ends the proof.
\end{proof}
\appendix
\section{Product rules for discrete Sobolev space}
\label{sec:prod}
The goal of this section is to derive the product rules for discrete Sobolev spaces. 

For $u\in H_N$, we define its discrete (complex) Fourier coefficients
\[\widehat{u}(k) = \frac{1}{N}\sum_{l=0}^{N-1} u(l)e^{-2i\pi lk/N},\quad -(N-1)/2\leq k\leq (N-1)/2.\]
These Fourier coefficient are related to the decomposition in the eigenbasis $(\varphi_{m,N},\psi_{m,N})$ by
\[|\widehat{u}(m)|^2 + |\widehat{u}(-m)|^2 = \langle u, \varphi_{m,N}\rangle^2 + \langle u, \psi_{m,N}\rangle^2,\quad 0\leq m\leq (N-1)/2.\]
Then, the discrete Sobolev norms can be expressed using the Fourier coefficients. 

\begin{lemme}
For all $\beta>1/2$, there exists $c>0$ such that for all $u\in H_N$, \[\|u\|_{\infty}\leq \|u\|_{H_N^\beta}.\]
\end{lemme}
\begin{proof}
For $0\leq j\leq N-1$, the inverse Fourier formula states that 
\[u(j) = \sum_{k=-(N-1)/2}^{(N-1)/2}\widehat{u}(k) e^{2i\pi kj/N}.\]
Then, from Cauhy-Schwartz inequality, we have
\begin{align*}
|u(j)|
&\leq \left(\sum_{k=-(N-1)/2}^{(N-1)/2}(1+|k|^2)^{\beta}|\widehat{u}(k)|^2\right)^{1/2}\left(\sum_{k=-(N-1)/2}^{(N-1)/2}(1+|k|^2)^{-\beta}\right)^{1/2}\\
&\leq c\left(\sum_{k=0}^{(N-1)/2}(1+|k|^2)^{\beta}\left(\langle u, \varphi_{k,N}\rangle^2 + \langle u, \psi_{k,N}\rangle^2\right)\right)^{1/2}.
\end{align*}
To conclude, the eigenvalue $\lambda_{k,N}$ are equivalent to $k^2$, uniformly in $N$. Therefore, we have proved that $|u(j)|\leq c\|u\|_{H_N^\beta}$, for all $0\leq j\leq N-1$.
\end{proof}

As an immediate consequence, we have a first product rule for $H_N^0\times H_N^\beta\to H_N^0$, for $\beta>1/2$, and its dual version, $H_N^0\times H_N^0\to H_N^{-\beta}$.

\begin{prop}\label{lmm:00b}
For all $\beta>1/2$, there exists $c>0$ such that for all $u,v\in H_N$ \[\|uv\|_{H_N^0}\leq c\|u\|_{H_N^0}\|v\|_{H_N^\beta},\quad\text{and}\quad\|uv\|_{H_N^{-\beta}}\leq c\|u\|_{H_N^0}\|v\|_{H_N^0}.\]
\end{prop}

\begin{proof}
The first inequality is direct. Let $u,v,\varphi\in H_N$, then, we have
\[\langle uv,\varphi\rangle \leq \|u\|_{H_N^0}\|v\varphi\|_{H_N^0}\leq c\|u\|_{H_N^0}\|v\|_{H_N^0}\|\varphi\|_{H_N^\beta}.\]
Hence the result.
\end{proof}

Then, we establish that $H_N^\beta$ is an algebra, for $\beta>1/2$. We follow the ideas from \cite{BF} and \cite{BK}. In order to do so, we introduce the trigonometric interpolation $I_N : H_N\to L^2$ as 
\[I_N(u) (x) = \sum_{k=-(N-1)/2}^{(N-1)/2} \widehat{u}(k) e^{2i\pi kx}.\]
It is the only trigonometric polynomial of degree smaller than $(N-1)/2$, which interpolates $u$ on the collocation points $j/N$, for $0\leq j\leq N-1$. The following lemma states that $I_N$ is continuous for the Sobolev norms.
  
\begin{lemme}
For all $0\leq \beta\leq 1$, there exists $c>0$ such that, for all $u\in H_N$, 
\[c^{-1}\|I_N(u)\|_{H^\beta}\leq\|u\|_{H_N^\beta}\leq c\|I_N(u)\|_{H^\beta}.\]
\end{lemme}
\begin{proof}
From Parseval' identity, we have $\|I_N(u)\|_{L^2} = \|u\|_{H_N^0}$. On the other hand,  we have $\|u\|_{H_N^1}^2 = \|u\|_{H_N^0}^2+ \|\nabla_N^+u\|_{H_N^0}^2$. For all $-(N-1)/2\leq k\leq (N-1)/2$, we have
\[\widehat{\nabla^+_Nu}(k) = N(e^{2i\pi k/N}-1)\widehat{u}(k).\]
Then, we have 
\begin{align*}
\|\nabla_N^+u\|_{H_N^0}^2 
&= \sum_{k=-(N-1)/2}^{(N-1)/2} \left|\widehat{\nabla^+_Nu}(k)\right|^2\\
&= \sum_{k=-(N-1)/2}^{(N-1)/2}|2\pi k|^2\left|\frac{e^{2i\pi k/N}-1}{2\pi k/N}\right|^2|\widehat{u}(k)|^2\\
\end{align*}
Hence, there exists $C>0$ such that  
\[C^{-1} \|\nabla_N^+u\|_{H_N^0}^2 \leq \sum_{k=-(N-1)/2}^{(N-1)/2}|2\pi k|^2|\widehat{u}(k)|^2\leq C \|\nabla_N^+u\|_{H_N^0}^2 .\]
To conclude on the $H_N^1$ case, let us remark that
\[ \sum_{k=-(N-1)/2}^{(N-1)/2}|2\pi k|^2|\widehat{u}(k)|^2 = \|\nabla I_N(u)\|^2_{L^2}.\]
Therefore, we have proved the result for $\beta = 0$ and $\beta=1$. The general case is obtained by linear interpolation.
\end{proof}

The upper bound $\beta\leq 1$ is artificial. In order to extend the result, up to $\beta = 2$, all we need to do is a control over $\|I_N(u)\|_{H^2}$. It resumes as controlling $\|\Delta_N u)\|_{H_N^0}$, with similar arguments. The interpolation $I_N$ allows working with classical Sobolev spaces, and their classical estimates. We obtain our second product rule.

\begin{prop}\label{lmm:bbb}
For all $1/2<\beta\leq 1$, there exists $c>0$ such that $u,v\in H_N$
\[\|uv\|_{H_N^\beta}\leq c \|u\|_{H_N^\beta}\|v\|_{H_N^\beta}.\]
\end{prop}
\begin{proof}
Let $u,v\in H_N$. Then, $I_N(u)I_N(v)$ is an interpolation of $uv$. Using \cite[Lemma 3.9]{BK}, for all $\beta>1/2$, there exists $c(\beta)$ such that \[\|uv\|_{H_N^\beta}\leq c(\beta) \|I_N(u)I_N(v)\|_{H^\beta}.\] 
Using a tame estimate, for classical Sobolev space, and our first lemma, we have 
\begin{align*}
\|I_N(u)I_N(v)\|_{H^\beta}
&\leq c\|I_N(u)\|_{\infty}\|I_N(v)\|_{H^\beta} + c\|I_N(u)\|_{H^\beta}\|I_N(v)\|_{\infty}\\
&\leq c \|I_N(u)\|_{H^\beta}\|I_N(u)\|_{H^\beta}\\
\end{align*} 
We conclude using the previous lemma.
\end{proof}

By bilinear interpolation between Lemma \ref{lmm:00b} and Lemma \ref{lmm:bbb}, we are now able to prove a product rule in $H_N^\alpha\times h_N^\beta$, for all $(\alpha,\beta)\in [0,1]^2$.


\begin{proof}[Proof of Theorem \ref{prop:prod+}]
Without any loss of generality, let us assume that $0\leq\alpha\leq \beta$.

$(i)$.  We have $1/2<\beta\leq 1$. For $\theta= (\beta-\alpha)/\beta\in[0,1]$, we have \[(\alpha, \beta) = \theta(0,\beta) +(1-\lambda)(\beta, \beta).\]
Form the previous lemmata, we have
\begin{equation*}
\left\{\begin{aligned}
\|uv\|_{H_N^0} &\leq C  \|u\|_{H_N^0}\|uv\|_{H_N^\beta}\\
\|uv\|_{H_N^\beta} &\leq C  \|u\|_{H_N^\beta}\|uv\|_{H_N^\beta}\\
\end{aligned}\right.
\end{equation*}
We use bilinear interpolation (see \cite{BL} exercise 5, chapter 3, \cite{LP}). Since  $[H_N,H_N^\beta]_{\alpha/\beta,2}$, the real interpolation space of $H_N$ and $H_N^\beta$, is $H_N^\alpha$ and since  $[H_N,H_N^\beta]_{\alpha/\beta,1}$ is embedded in $H_N^\gamma$ for all $\gamma< \theta\times0 + (1-\theta)\beta = \alpha$, we have
\[\|uv\|_{H_N^\gamma} \leq C  \|u\|_{H_N^\alpha}\|uv\|_{H_N^\beta}.\]

$(ii)$ Let $0<\varepsilon$, and $\theta = \frac{2\beta}{1+\varepsilon}$. We have 
\[(\alpha, \beta) = \theta\left(\frac{\alpha}{\beta}\frac{1+\varepsilon}{2}, \frac{1+\varepsilon}{2}\right) +(1-\theta)(0,0).\]
Let $\gamma=\alpha(1+\varepsilon)/(2\beta)-\varepsilon$. From  Lemma \ref{lmm:00b} and point $(i)$, we have
\begin{equation*}
\left\{\begin{aligned}
&\|uv\|_{H_N^\gamma} \leq C  \|u\|_{H_N^{\frac{2\alpha}{\beta(1+\varepsilon)}}}\|uv\|_{H_N^{\frac{1+\varepsilon}{2}}}\\
&\|uv\|_{H_N^{-\frac{1+\varepsilon}{2}}} \leq C  \|u\|_{H_N^0}\|uv\|_{H_N^0}\\
\end{aligned}\right.
\end{equation*}
Then for all $\eta< \theta\gamma - (1-\theta)(1+\varepsilon)/2 = \alpha+\beta-1/2-\varepsilon(1/2+2\beta/(1+\varepsilon))$, we have, again by bilinear interpolation,
\[\|uv\|_{H_N^\eta} \leq C  \|u\|_{H_N^\alpha}\|uv\|_{H_N^\beta}.\]
This ends the proof.
\end{proof}

To conclude, we obtain dual version, in $H_N^{-\alpha}\times H_N^\beta$.

\begin{proof}[Corollary \ref{prop:prod+}]
For all $\varphi\in H_N$, we have
\[\langle uv, \varphi\rangle \leq \|u\|_{H_N^{-\alpha}} \|v\varphi\|_{H_N^\alpha}.\]
From point $(i)$ from the previous lemma, for all $0\leq \beta,\gamma$ such that $0<\alpha< \beta\wedge\gamma$ and $1/2<\beta\vee\gamma$, we have
\[\langle uv, \varphi\rangle \leq \|u\|_{H_N^{-\alpha}}\|v\|_{H_N^\beta}\|\varphi\|_{H_N^\gamma}.\]
This proves point $(i)$. The second point is prove similarly, using the product rule $(ii)$.
\end{proof}

\section{Existence and uniqueness of the solution}
\label{sec:existunic}
In this section, we prove that for every initial data $(u(0),v(0)$ is $(L^\infty\cap H^\beta)\times H^{-\alpha}$, there exists a unique solution to Equation \eqref{eq:uv}.

\subsection*{Existence} 
The proof of the existence of a solution follows the same steps as Proposition \ref{prop_w}. We define the truncated field $R_C^M$ and $R_D^M$, with compactly supported $b_C^M$ and $b_D^M$. Let $\TT =(\TT_1,\TT_2)$ be the operator defined on $\CC\left([0,T], (L^\infty\cap H^\beta)\times H^{-\alpha}\right)$ as 
\begin{align*}
\TT_1(u,v)(t) =& e^{\Delta t}u_0 + \int_0^t e^{\Delta (t-s)} R_C^M(u(s), v(s))\, ds\\
\TT_2(u,v)(t) =& v_0 + \int_0^t  R_D^M(u(s), v(s))\, ds\\
\end{align*} 
First, we remark that $\CC\left([0,T], (L^\infty\cap H^\beta)\times H^{-\alpha}\right)$ is stable under $\TT$. Then, we show that for $0<T_0\leq T$ small enough, the restriction of $\TT$ to $\CC([0,T_0])$ is a contraction. This allows to construct a solution $(u,v)$ on $[0,T_0]$ and on $[0,T]$ by iteration. To conclude, according to Assumption \ref{ass:RC} and the maximum principle, $u$ is bounded in $\L^\infty$ by $\max\{M, \|u(0)\|_\infty\}$. It follows that $b_C^M(u) = b_C(u)$, and  $b_D^M(u) = b_D(u)$. So $(u,v)$ is also a solution for the initial problem, without truncation.

\subsection*{Uniqueness} 
Let $(u_1,v_1), (u_2, v_2)\in\CC\left([0,T], (L^\infty\cap H^\beta)\times H^{-\alpha}\right)$, be two solutions of Equation \eqref{eq:uv}, with the same initial data $(u_0,v_0)$. Let $0<\varepsilon,\gamma$, such that $1/2+\varepsilon+\gamma<2$ and $-\gamma<1/2+\varepsilon-\alpha-1/2$. For all $0\leq t\leq T$, we have

\begin{align*}
\|u_1(t)-u_2(t)\|_{L^\infty\cap H^\beta}  
\leq & \|u_1(t)-u_2(t)\|_{H^{1/2+\varepsilon}} \\
\leq & C\int_0^t (t-s)^{-(1/2+\varepsilon+\gamma)/2}\|u_1(s)\|_{H^{\beta}}\|v_1(s)-v_2(s)\|_{H^{-\alpha}}\, ds\\
& + C\int_0^t (t-s)^{-(1/2+\varepsilon+\gamma)/2}\|(u_1(s)-u_2(s)\|_{H^{\beta}}\|v_2(s)\|_{H^{-\alpha}}\, ds\\
& + C\int_0^t (t-s)^{-(1/2+\varepsilon+\alpha)/2}\|v_1(s)-v_2(s)\|_{H^{-\alpha}}\, ds\\
& +C\int_0^t (t-s)^{-(1/2+\varepsilon)/2}\|u_1(s)-u_2(s)\|_{\infty}\, ds\\
\leq & + C\int_0^t (t-s)^{-(1/2+\varepsilon+\gamma)/2}\|(u_1(s)-u_2(s)\|_{H^{\beta}}\, ds\\
& + C\int_0^t (t-s)^{-(1/2+\varepsilon+\alpha)/2}\|v_1(s)-v_2(s)\|_{H^{-\alpha}}\, ds\\
& +C\int_0^t (t-s)^{-(1/2+\varepsilon)/2}\|u_1(s)-u_2(s)\|_{\infty}\, ds\\
\end{align*} 
Lastly, we control the difference between $v_1$ and $v_2$ with
\[\|v_1(t)-v_2(t)\|_{H^{-\alpha}}
\leq   C\int_0^t \|v_1(s)-v_2(s)\|_{H^{-\alpha}}\, ds
 +C\int_0^t \|u_1(s)-u_2(s)\|_{\infty}\, ds\]
Then, we have proved that
\begin{align*}
&\|(u_1(t),v_1(t))-(u_2(t), v_2(t))\|_{(L^\infty\cap H^\beta)\times H^{-\alpha}} \\
&\quad\le C\int_0^t(t-s)^{-(1/2+\varepsilon+\gamma)/2}\|(u_1(s),v_1(s))-(u_2(s), v_2(s))\|_{(L^\infty\cap H^\beta)\times H^{-\alpha}}\, ds.
\end{align*} 
The proof of uniqueness is concluded by a singular version of Gr\"onwall's inequality, as in Proposition \ref{prop:uNwN}.

\bibliographystyle{abbrv}
\bibliography{biblio}
\end{document}